\documentclass[9 pt, twocolumn]{IEEEtran}
\usepackage{amsmath,amssymb,euscript,yfonts,psfrag,latexsym,dsfont,graphicx}
\usepackage{bbm,color,amstext,wasysym,subfig,flushend,parskip,balance}
\usepackage{amsthm}
\usepackage{algorithm}
\usepackage{algorithmic}
\usepackage{tikz}
\usetikzlibrary{fit,positioning,shapes,arrows}
\graphicspath{{./},{./figures/}}

\newtheorem{thm}{Theorem}

\newtheorem{prop}{Proposition}

\newtheorem{remark}[thm]{Remark}

\newcommand{\E}{{\mathbb E}}  % this is for expectation -- we can change
\newcommand{\mE}{{\mathbb E}}

\newcommand{\mR}{{\mathbb R}}
\newcommand{\R}{{\mathbb R}}

\newcommand{\cH}{{\mathcal H}}

\newcommand{\cN}{{\mathcal N}}
\newcommand{\cO}{{\mathcal O}}
\newcommand{\cP}{{\mathcal P}}
\newcommand{\cQ}{{\mathcal Q}}

\newcommand{\cV}{{\mathcal V}}

\newcommand{\cY}{{\mathcal Y}}

\newcommand{\bB}{{\mathbf B}}
\newcommand{\bC}{{\mathbf C}}
\newcommand{\bK}{{\mathbf K}}
\newcommand{\bM}{{\mathbf M}}
\newcommand{\bU}{{\mathbf U}}
\newcommand{\bm}{{\mathbf m}}

\newcommand{\bu}{{\mathbf u}}
\newcommand{\bx}{{\mathbf x}}

\newcommand{\HH}{{\mathrm H}}

\newcommand{\vv}{v}
\newcommand{\tstart}{0}
\newcommand{\tend}{1}
\newcommand{\x}{x}
\newcommand{\nrho}{\rho}
\newcommand{\ett}{{\bf 1}}

\newcommand{\tr}{\operatorname{Tr}}

\newcommand{\argmin}{\operatorname{argmin}}

\definecolor{grey}{rgb}{0.6,0.6,0.6}
\definecolor{lightgray}{rgb}{0.97,.99,0.99}

\begin{document}
\title{Density control of interacting agent systems}

\author{Yongxin Chen
\thanks{This work was supported by the NSF under grant 1901599, 1942523 and 2008513.}
\thanks{Y.\ Chen is with the School of Aerospace Engineering, Georgia Institute of Technology, Atlanta, GA, USA. {\tt\small \{yongchen\}@gatech.edu}}}
%\markboth{\today}{}

\maketitle

\begin{abstract}
We consider the problem of controlling the group behavior of a large number of dynamic systems that are constantly interacting with each other. These systems are assumed to have identical dynamics (e.g., birds flock, robot swarm) and their group behavior can be modeled by a distribution. Thus, this problem can be viewed as an optimal control problem over the space of distributions. We propose a novel algorithm to compute a feedback control strategy so that, when adopted by the agents, the distribution of them would be transformed from an initial one to a target one over a finite time window. Our method is built on optimal transport theory but differs significantly from existing work in this area in that our method models the interactions among agents explicitly. From an algorithmic point of view, our algorithm is based on a generalized version of proximal gradient descent algorithm and has a convergence guarantee with a sublinear rate. We further extend our framework to account for the scenarios where the agents are from multiple species. In the linear quadratic setting, the solution is characterized by coupled Riccati equations which can be solved in closed-form. Finally, several numerical examples are presented to illustrate our framework. 
\end{abstract}

\section{Introduction}
%\cite{BacConGenLeo20}
%What's density control. Why is it important? Applications?
%
%Optimal transport approach, Schr\"odinger bridge
%
%Limitation of optimal transport, no interactions
%
%The objective of this paper
%
%Our approach, theory, algorithm, LQG and experiment

Consider the swarm control \cite{BraFerBirDor13,ChuParDamKum18} task to establish and regulate a formation of $N$ drones (``agents''). There are two substantially different angles from which one may consider this problem. A first (straightforward) approach is to concatenate the states of all the drones into one state vector and then formulate a control problem over this joint state space. Suppose the state dimension of each individual drone is $d$, then the dimension of the joint state space becomes $Nd$, which scales linearly as the group size $N$ increases. An alternative approach is to treat the distribution of the drones as {\em the} state of a system, and formulate a corresponding optimal control problem over the space of distributions. One major difference between the two approaches is that, in the former, each individual has a label and the controller aims to jointly optimize the performance of each individual, while in the latter, the {\em individuals are indistinguishable and only group behavior matters}. Thus, when the optimality criteria only involves the group behavior of individuals, the problem reduces to a density control problem. In this formulation, the state is a probability distribution and is independent of the group size $N$.
% One caveat is that the state dimension becomes infinite. 

The density control problem is an optimal control problem over distributions where the objective/cost function is fully determined by the evolution of the distribution. The dynamics of the distribution of the group follows the Liouville equation, or the Fokker-Planck equation if the individual dynamics is stochastic, or the McKean-Vlasov equation if the individuals interact with each other \cite{Mck66,HuaMalCai06}. In addition to controlling the group behavior of a large number of individuals, the density control can also be used to deal with controlling the state uncertainty of a single dynamical system. When a dynamical system either has uncertainty in its initial state or is disturbed by random process noise, the state remains uncertain and can be captured by a probability distribution at each time point. Regulating the state uncertainties of such systems is thus equivalent to controlling their state distribution. The density control problem provides a more direct approach to achieve this goal than standard stochastic control theory where an indirect cost needs to be properly handcrafted. The density control framework has found applications in a range of areas \cite{ZhuGriSke95,RidTsi18,OkaTsi19,Kal02,SinManBra21}.

Two popular tools for density control are the optimal transport (OT) theory \cite{Vil03} and the Schr\"odinger bridge theory \cite{Leo13,CheGeoPav21a}. The latter can be viewed as a regularized version of the former. The minimum effort control between two specified distributions over a finite time interval can be addressed using the optimal transport theory if the dynamics is deterministic and the Schr\"odinger bridge theory if the dynamics is stochastic. This paradigm is suitable for controlling uncertainty of a single dynamical system. An important instance of it is covariance control \cite{HotSke87,CheGeoPav14a,CheGeoPav14b,HalWen16,CheGeoPav17a,Bak18,RidTsi18,OkaTsi19} where the distribution at each time point is parametrized by a Gaussian distribution determined by its mean and covariance. The connections between covariance control and the Schr\"odinger bridge has been extensively studied in \cite{CheGeoPav14a}. One major limitation of these existing methods of density control for controlling group behaviors is that the individuals in the group are assumed to be independent to each other. Thus, important properties for swarm control such as collision avoidance are not explicitly modeled in these methods. Indeed, when these methods are applied for swarm control \cite{KriMar18,ElaBer19,InoItoYos20,BisElaBer20,ZheHanLin20,SinManBra21}, the collision avoidance requirement is often ignored. 

%The existing methods for distribution control are limited to controlling the population of one type of individuals, and the individuals have to be independent to each other. 

The goal of this work is to address density control problems for controlling group behaviors when the individuals in the group are constantly interacting with each other. We consider the scenario where the interactions between every two agents are through an interactive potential that is the same for each pair of agents. When the number of individuals is large, the distribution of the individuals solves the McKean-Vlasov equation. We propose to build on recent work on mean-field Schr\"odinger bridges \cite{BacConGenLeo20} which is a generalization of the Schr\"odinger bridge theory when the prior dynamics is modeled by a McKean-Vlasov equation, and reformulate this problem into an optimization over the space of path measures. With proper discretization over space and time the problem becomes a nonlinear version of the multi-marginal optimal transport (MOT) \cite{Pas15,HaaRinChe20,HaaSinZhaChe20}. To numerically solve this problem, we adopt the proximal gradient algorithm \cite{BecTeb03,Bec17} to sequentially linearize the nonlinear MOT and then take advantage of existing algorithms \cite{HaaRinChe20,HaaSinZhaChe20} for MOT for each iteration. We further extend our method to the setting when multiple species are involved. Finally, in the linear quadratic setting where the dynamics is linear and the cost function is quadratic, we characterize the optimal solution via a coupled Riccati equation system and obtain the closed-form solution. 

%We plan to establish a distribution control framework that can handle (i) populations of multiple species and (ii) individuals that interact with each other. For (i) we propose to build on our recent work on graphical optimal transport \cite{HaaSinZhaChe20}. For (ii), we propose to extend the recent work on mean-field Schr\"odinger bridges \cite{BacConGenLeo20} to more general dynamical systems. We will also consider (iii) a game-theoretic setting of the distribution control problem to account for the interaction between two groups or between a group and an individual. This can be used for interactive swarm control. Finally, we plan to develop (iv) a scalable algorithm for our distribution control framework. We propose to combine the algorithm in \cite{HaaSinZhaChe20} that exploits the cost structure of the problem and the algorithm in \cite{CalHal19} that relies on samples instead of brutal force discretizations. 

The rest of the paper is organized as follows. In Section \ref{sec:back} we briefly introduce the several tools that will be used in this work. The main results on density control for interacting agent system are presented in Section \ref{sec:DC}. An extension to density control problems involving multiple species is provided in Section \ref{sec:multi}. We investigate the problem in the linear quadratic setting in Section \ref{sec:LQ} and obtain a closed-form solution. Several numerical examples are presented in Section \ref{sec:eg} to illustrate the proposed framework. This is followed by a short concluding remark in \ref{sec:conclusion}.

\section{Background}\label{sec:back}
In this section we introduce several mathematical tools on which our density control framework is based, including optimal transport and its generalizations, and the proximal gradient algorithm. 
\subsection{Optimal transport}
Given two nonnegative measures $\mu, \nu$ on ${\mathbb R}^d$ having equal total mass (often assumed to be probability distributions), the Monge's formulation of optimal transport seeks a transport map
\[
T\;:\;{\mathbb R}^d\to{\mathbb R}^d\;:\;x\mapsto T(x)
\]
from $\mu$ to $\nu$ in the sense $T_\sharp \mu=\nu$, that incurs minimum cost of transportation
\[
\int c(x,T(x))\mu(dx).
\]
Here, $c(x,y)$ stands for the transportation cost per unit mass from point $x$ to $y$.
The dependence of the total transportation cost on $T$ is highly nonlinear, complicating early analyses to the problem \cite{Vil03}. This problem was later relaxed by Kantorovich, where, instead of a transport map, a joint distribution $\pi$ on the product space $\R^d\times\R^d$ is sought.
Let $\Pi(\mu,\nu)$ be the set of joint distributions of $\mu$ and $\nu$, then the Kantorovich formulation of OT reads
    \begin{equation}\label{eq:OptTrans}
        \inf_{\pi\in\Pi(\mu,\nu)}\int_{\R^d\times\R^d}c(x,y)\pi(dxdy).
    \end{equation}

Both the Monge's and the Kantorovich's formulations are ``static'' focusing on ``what goes where.''
It turns out that the OT problem can also be cast as a dynamical problem with a temporal dimension. In particular, when $c(x,y)=\frac{1}{2}\|x-y\|^2$, OT can be formulated as a stochastic control problem
    \begin{subequations}\label{eq:stochcontr}
    \begin{eqnarray}\label{eq:stochcontr1}
        &&\inf_{\vv\in\cV}\E\left\{\int_{\tstart}^{\tend}\frac{1}{2}\|\vv(t,\x^\vv(t))\|^2dt\right\},
        \\
        && \dot{\x}^\vv(t)=\vv(t,\x^\vv(t)),\label{eq:stochcontr2}
        \\
        && \x^v(\tstart)\sim\mu,\quad \x^v(\tend)\sim\nu.
    \end{eqnarray}
    \end{subequations}
Here $\cV$ represents the family of admissible state feedback control laws. Note that this control problem \eqref{eq:stochcontr} differs from standard ones in that the terminal constraint $\x^v(\tend)\sim\nu$, meaning $\x^v(1)$ follows distribution $\nu$, is unconventional. In \eqref{eq:stochcontr}, the goal is to find an optimal control policy to drive the system \eqref{eq:stochcontr2} from an uncertain initial state $\x^v(\tstart)\sim\mu$ to an uncertain target state $\x^v(\tend)\sim\nu$. The solution to \eqref{eq:stochcontr} specifies how to move mass over time from configuration $\mu$ to $\nu$, providing more resolution to the optimal transport plan.

Assuming $x^v(t)$ has a absolutely continuous distribution with density $\rho_t$, $\rho_t$ satisfies weakly\footnote{In the sense that $\int_{[0,1]\times\mR^d}[(\partial_t f+v\cdot\nabla f)\rho_t]dtdx=0$ for smooth functions $f$ with compact support.} the continuity equation
    \begin{equation}\label{eq:continuity}
        \partial_t\rho_t+\nabla\cdot(\vv\rho_t)=0,
    \end{equation}
and the total transport cost becomes
    \[
        \E\left\{\int_{\tstart}^{\tend}\frac{1}{2}\|\vv(t,\x^\vv(t))\|^2dt\right\}=\int_{\R^d}\int_{\tstart}^{\tend}\frac{1}{2}\|\vv(t,x)\|^2\rho_t(x) dtdx.
    \]
Thus, \eqref{eq:stochcontr} is equivalent to \cite{BenBre00}
    \begin{subequations}\label{eq:BB}
    \begin{eqnarray}\label{eq:BB1}
        &&\inf_{\rho,v}\int_{\R^d}\int_{\tstart}^{\tend}\frac{1}{2}\|\vv(t,x)\|^2\rho_t(x) dtdx,\\&&\partial_t\rho_t+\nabla\cdot(\vv\rho_t)=0,\label{eq:BB2}\\&& \rho_0=\mu, \quad \rho_1=\nu.\label{eq:boundary}
    \end{eqnarray}
    \end{subequations}
The minimum is taken over all the pairs $\rho, v$ satisfying
 \eqref{eq:BB2}-\eqref{eq:boundary} and some other technical conditions,
 see \cite[Theorem 8.1]{Vil03}, \cite[Chapter 8]{AmbGigSav06}.
 
Suppose we have a large number of individuals that share the same dynamics \eqref{eq:stochcontr2} but are independent to each other. Assume their initial states follow the same distribution $\mu$, then \eqref{eq:BB} can be viewed as an density control problem for this group whose objective is to find a common control strategy for the group so that it would reach target distribution/configuration $\nu$ at time $t=1$. The solution to \eqref{eq:BB} is characterized by the coupled partial differential equations (PDEs) 
%	\begin{subequations}
	\begin{eqnarray*}
		&& \partial_t \lambda + \frac{1}{2} \nabla \lambda ^T \nabla \lambda  = 0
		\\&&
		\partial_t \rho_t +\nabla\cdot(\rho_t \nabla\lambda) = 0
		\\&&
		\rho_{0} = \mu, \quad \rho_{1} = \nu,
	\end{eqnarray*}
%	\end{subequations}
and the optimal control is $v(t,x) = \nabla \lambda(t,x)$.

\subsection{Schr\"odinger bridges}
In 1931/32, Schr\"odinger \cite{Sch31,Sch32} posed the following problem: A large number N of independent Brownian particles in $\mR^d$ is observed to have an empirical distribution approximately equal to $\mu$ at time $t=0$, and at some later time $t=1$ an empirical distribution approximately equal to $\nu$. Suppose that $\nu$ differs from what it should be according to the law of large numbers, namely
    \[
        \int q_\epsilon(0,x,1,y)\mu(dx),
    \]
where
    \[
        q_\epsilon(s,x,t,y)=(2\pi)^{-d/2}[\epsilon(t-s)]^{-d/2}\exp\left(-\frac{\|x-y\|^2}{2\epsilon(t-s)}\right)
    \]
denotes the scaled Brownian transition probability density. It is apparent that the particles have been transported in an unlikely way. But of the many unlikely ways in which this could have happened, which one is the most likely?

This problem can be understood in the modern language of large deviation theory as a problem \cite{Fol88} of determining a probability law $\cP$ on the path space $\Omega = C([0,1],\mR^d)$ that minimizes the relative entropy (a.k.a., Kullback-Leibler divergence)\footnote{$\frac{d\cP}{d\cQ}$ denotes the Radon-Nikodym derivative between $\cP$ and $\cQ$.}
\begin{equation}\label{eq:SBmeasure}
{\rm KL}(\cP \|\cQ):=\int_\Omega \log\left(\frac{d\cP}{d\cQ}\right)d\cP.
\end{equation}
Here $\cQ$ is the probability law induced by the Brownian motion and $\cP$ is chosen among probability laws that are absolutely continuous with respect to $\cQ$ and have the prescribed marginals. 
The solution to this optimization problem is referred to as the {\em Schr\"odinger bridge}. Existence and uniqueness of the minimizer have been proven in various degrees of generality by Fortet \cite{For40}, Beurling \cite{Beu60}, Jamison \cite{Jam74}, F\"ollmer \cite{Fol88}. 

It has been shown that the above Schr\"odinger's problem can be reformulated as the stochastic control problem \cite{CheGeoPav14e}
    \begin{subequations}\label{eq:stochsticcontrol}
    \begin{eqnarray}
        && \inf_{v\in \cV} \mE \left\{\int_{\tstart}^{\tend}\frac{1}{2}\|v(t,X_t)\|^2dt\right\},
        \\
        && dX_t=v(t, X_t)dt+\sqrt{\epsilon}dB_t,\label{eq:stochasticcontrol1}
        \\
        && X_0 \sim \mu, \quad X_1 \sim \nu.
    \end{eqnarray}
    \end{subequations}
Here $\cV$ is the class of finite energy Markov controls.
This reformulation of the Schr\"odinger problem relies on the fact that the relative entropy between distributions induced by the controlled and uncontrolled processes is
	\[
		{\rm KL} (\cP\| \cQ) = \mE \left\{\int_{\tstart}^{\tend}\frac{1}{2\epsilon}\|v(t,X_t)\|^2dt\right\}.
	\]
The proof is based on Girsanov theorem, see \cite{Dai91,KarShr88}. It is easy to see that \eqref{eq:stochsticcontrol} has the following density control reformulation \cite{CheGeoPav14c,GenLeoRip15}
    \begin{subequations}\label{eq:daipra}
    \begin{eqnarray}\label{eq:expectedcost}
        && \inf_{\nrho,v} \int_{\mR^d}\int_{\tstart}^{\tend}\frac{1}{2}\|v(t,x)\|^2\nrho_t dtdx,
        \\
        && \partial_t \nrho_t+\nabla\cdot(v\nrho_t)- 
        \frac{\epsilon}{2}\Delta \nrho_t=0,\label{eq:fokkerplanck}
        \\
        && \nrho_0=\mu, \quad \nrho_1=\nu,
    \end{eqnarray}
    \end{subequations}
where the infimum is over smooth fields $v$ and $\rho$ that solve weakly of the  corresponding Fokker-Planck equation \eqref{eq:fokkerplanck}. 

Formulation \eqref{eq:daipra} resembles the OT problem \eqref{eq:BB} except for the presence of the Laplacian in \eqref{eq:fokkerplanck}.
It has been shown \cite{Mik04,MikThi08,Leo12,Leo13} that the OT problem is, in a suitable sense, indeed the limit of the Schr\"{o}dinger problem when the diffusion coefficient $\epsilon$ of the reference Brownian motion goes to zero. On the other hand, the Schr\"odinger bridge can be viewed as a regularized version of OT. Similar to \eqref{eq:BB}, the solution to \eqref{eq:daipra} is characterized by the coupled PDEs
	\begin{subequations}
	\begin{eqnarray}
		&& \partial_t \lambda + \frac{1}{2} \nabla \lambda ^T \nabla \lambda +\frac{\epsilon}{2}\Delta \lambda = 0
		\\&&
		\partial_t \rho_t +\nabla\cdot(\rho_t \nabla\lambda) -\frac{\epsilon}{2} \Delta \rho_t= 0
		\\&&
		\rho_{0} = \mu, \quad \rho_{1} = \nu,
	\end{eqnarray}
	\end{subequations}
and the corresponding optimal control policy is $v(t,x) = \nabla \lambda(t,x)$.

\subsection{Multi-marginal optimal transport}
In this section we introduce discrete OT where the supports of the marginal distributions are discrete sets. 
 In this discrete setting, the marginals $\boldsymbol\mu_1\in \mR_+^{d_1},\boldsymbol\mu_2 \in \mR_+^{d_2}$ are nonnegative vectors with equal sum. 
 The transport cost function can be rewritten in the matrix form $\bC=[C(x_1,x_2)]\in \mR^{d_1\times d_2}$ where $C(x_1,x_2)$ represents the transport cost of moving a unit mass from point $x_1$ to $x_2$. Similarly, a transport plan is encoded in a joint probability matrix $\bB = [B(x_1,x_2)] \in \mR_+^{d_1\times d_2}$ of $\boldsymbol\mu_1, \boldsymbol\mu_2$. The total transport cost is $\sum_{x_1,x_2} C(x_1,x_2) B(x_1,x_2)= \tr(\bC^T \bB)$ and the OT problem becomes
\begin{equation}\label{eq:omt_bi_discrete}
    \min_{\bB \in \mR_+^{d_1\times d_2}}~ \tr(\bC^T \bB) 
\quad \text{ subject to }\quad \bB \ett = \boldsymbol\mu_1,~ \bB^T \ett = \boldsymbol\mu_2,
\end{equation}
where $\ett$ denotes a vector of ones of proper dimension. The constraints are to enforce that $\bB$ is a joint distribution of $\boldsymbol\mu_1$ and $\boldsymbol\mu_2$. 

Multi-marginal optimal transport (MOT) extends OT to the setting involving multiple distributions. In particular, in MOT, one seeks a transport plan among a set of marginals $\boldsymbol\mu_1,\dots,\boldsymbol\mu_J$ with $J\geq 2$. 
In the discrete setting, the transport cost is encoded in a tensor $\bC = [C(x_1,x_2,\ldots, x_J)]\in \mR^{d_1\times d_2\times\cdots\times d_J}$ where $C(x_1,x_2,\ldots, x_J)$ denotes the unit cost associated with $(x_1, x_2,\ldots, x_J)$, and the transport plan is described by a tensor  $\bB \in \mR_+^{d_1\times d_2\times\cdots\times d_J}$.
For a transport plan $\bB$, the total cost is $\langle \bC, \bB\rangle := \sum_{x_1,x_2,\ldots, x_J} C(x_1,\ldots, x_J) B(x_1,\ldots, x_J)$.
Thus, similar to \eqref{eq:omt_bi_discrete}, MOT has a linear programming formulation
\begin{equation} \label{eq:omt_multi_discrete}
\min_{\bB \in \mR_+^{d_1\times \dots \times d_J}} \langle \bC, \bB \rangle  \quad
\text{ subject to } ~ P_j (\bB) = \boldsymbol\mu_j,  \text { for } j \in \Gamma,
\end{equation}
where $\Gamma\subset \{1,2,\dots,J\}$ is an index set specifying which marginal distributions are given, and the projection on the $j$-th marginal of $\bB$ is defined by
\begin{equation} \label{eq:proj_discrete}
P_j(\bB) = \sum_{x_1,\ldots,x_{j-1},x_{j+1},\ldots,x_J} B (x_1,\dots,x_{j-1},x_j,x_{j+1},\dots,x_J).
\end{equation}

A popular method to solve the OT problem is entropy regularization, which adds an entropy term
\begin{equation}
\cH(\bB) = - \sum_{x_1,\dots,x_J} B(x_1,\dots,x_J)  \log B(x_1,\dots,x_J)
\end{equation}
to \eqref{eq:omt_multi_discrete}, resulting in the strictly convex optimization problem
\begin{equation} \label{eq:omt_multi_regularized}
\min_{\bB \in \mR^{d_1\times \dots \times d_J}} \langle \bC, \bB \rangle - \epsilon \cH(\bB)
~ \text{ subject to } ~ P_j (\bB) = \boldsymbol\mu_j,  \text { for } j \in \Gamma
\end{equation}
with $\epsilon>0$ being a regularization parameter. 
Invoking Lagrangian duality, one can show that the optimal solution to \eqref{eq:omt_multi_regularized} is
\begin{equation}\label{eq:optB}
\bB = \bK \odot \bU,
\end{equation}
where $\odot$ denotes element-wise multiplication, 
	\begin{equation}\label{eq:K}
		\bK = \exp(- \bC/\epsilon),
	\end{equation}
and $\bU= \bu_1 \otimes \bu_2 \otimes \dots \otimes \bu_J$ with the vectors $\bu_j \in \mathbb{R}^{d_j}$ being associated with the Lagrange multipliers.

The Sinkhorn algorithm \cite{Sin64,Cut13,FraLor89} iteratively updates the vectors $\bu_j$ according to
\begin{equation} \label{eq:sinkhorn_multi}
\bu_j \leftarrow \bu_j \odot \boldsymbol\mu_j ./ P_j(\bK \odot \bU),
\end{equation}
for all $j\in\Gamma$. Here $./$ denotes element-wise division. 
%-----------------------------------------------------------
\begin{algorithm*}[tb]
   \caption{Sinkhorn Algorithm for MOT}
   \label{alg:sinkhorn}
\begin{algorithmic}
    \STATE Compute $\bK=\exp(- \bC/\epsilon)$
   \STATE Initialize $\bu_1, \bu_2, \ldots, \bu_J$ to $\mathbf{1}$
   \WHILE{not converged}
   \FOR{$j \in \Gamma$}
        \STATE Compute $\bU= \bu_1 \otimes \bu_2 \otimes \dots \otimes \bu_J$
        \STATE Update $\bu_j$ as $\bu_j \leftarrow \bu_j \odot \boldsymbol\mu_j ./ P_j(\bK \odot \bU)$
    \ENDFOR
    \ENDWHILE
\end{algorithmic}
\end{algorithm*}
%-----------------------------------------------------------
The Sinkhorn algorithm (Algorithm~\ref{alg:sinkhorn}) has a global convergence guarantee \cite{BauLew00,LuoTse93}. Moreover, it has a linear convergence rate \cite{LuoTse92,HaaRinChe20}.
%Sinkhorn belief propagation
%Collective Forward Backward algorithm

Regardless the rapid developments of OT algorithms, computation remains a major problem preventing OT, in particular MOT, from widely used in applications. Although Algorithm~\ref{alg:sinkhorn} is easy to implement and considerably faster than generic linear programming solvers, its complexity still scales exponentially as $J$ grows. The computational bottleneck of it lies in the calculation of the projections $P_j(\bB)$, $j \in \Gamma$ in \eqref{eq:proj_discrete}. 

Recently it was discovered that the computation of MOT can be greatly accelerated if the cost tensor $\bC$ has a graphical structure \cite{HaaSinZhaChe20}, that is, the cost tensor $\bC$ can be decomposed as
\begin{equation}\label{eq:cost_structure}
    C(\bx) = C(x_1,x_2,\ldots,x_J) = \sum_{(i,j)\in E} C_{ij}(x_i,x_j),
\end{equation}
where $E$ denotes the set of edges of an undirected graph $G = (V,E)$. An important instance of this graphical OT is the Barycenter problems \cite{AguCar11,KuaTab19,FanTagChe20} where the cost $\bC$ can be decomposed into the sum of pairwise costs between the target distribution and each given marginal distribution; the cost thus corresponds to a star-shaped tree.
%	\begin{figure}[h]
%	\centering
%	\includegraphics[width=0.27\textwidth]{sbp_explain1}
%	\caption{An example of graphical OT}
%	\label{fig:general_factor_clean}
%	\end{figure}
	
	\begin{figure}[h]
	\centering
	\begin{tikzpicture}
%\tikzstyle{block} = [draw, shape=rectangle, minimum height=2.5cm, minimum width= 3cm , line width=0.6pt, text centered,  text width=2.5cm]
%\tikzstyle{block1} = [draw, shape=rectangle, minimum height=1cm, minimum width= 1cm , line width=0.4pt, text centered,text width=1.2cm]
%
%    \tikzstyle{block2} = [draw, shape=rectangle, minimum height=1.3cm, minimum width=3.4cm, line width=0.6pt,text centered, text width=3.1cm]
%    
%    \tikzstyle{sum_block} = [draw, circle, minimum size=1.4em, inner sep=1pt]
     
       \tikzstyle{bluenode}=[draw, circle, minimum size=2em, inner sep=1pt, fill=blue!45];
       \tikzstyle{bluenode1}=[draw, circle, minimum size=2em, inner sep=1pt, fill=blue!4];
        % Change vertex color and intensity here
%        \tikzstyle{blacknoder}=[draw, circle, minimum size=0.4em, inner sep=1pt, fill=red];
%        \tikzstyle{blacknodeg}=[draw, circle, minimum size=0.4em, inner sep=1pt, fill=green];
  \tikzstyle{edge}=[>=stealth, thick, black] % Change edge color and intensity here
  
  \tikzstyle{ref}=[draw, circle, minimum size=0.2em, inner sep=1pt, color = white, fill=white];

    %Creating Blocks and Connection Nodes
    
\node [bluenode] at (-0.8,-2) (b1) {$x_1$};

\node [bluenode1] at (0,0) (b4) {$x_4$};

\node [bluenode] at (4,0) (b2) {$x_2$};

\node [bluenode] at (4,2) (b3) {$x_3$};

\node [bluenode1] at (0,1.5) (b5) {$x_5$};

\node [bluenode1] at (2,-1.5) (b6) {$x_6$};

\node [bluenode1] at (2,1.75) (b7) {$x_7$};

%\draw[edge, dashed] (0,0) to (b2);
\draw[edge, ] (b2) to (b6);
\draw[edge, ] (b6) to (b4);
\draw[edge, ] (b4) to (b1);
\draw[edge, ] (b4) to (b5);
\draw[edge, ] (b7) to (b5);
\draw[edge, ] (b3) to (b7);

%\draw[edge, -triangle 45, bend left] (b4) to (b6);
%\draw[edge, -triangle 45, bend right] (b6) to (b2);
%
%\node[text width=1.8 cm,  black] at (1.8, -0) {$m_{4 \rightarrow 6}(x_6)$};
%\node[text width=1.9 cm,  black] at (3.8, -1.5) {$m_{6 \rightarrow 2}(x_2)$};
%\node[text width=0.5cm,  black] at (-0.8, -2.6) {$\mathbf{y}_1$};
%\node[text width=0.5cm,  black] at (4.8, 0) {$\mathbf{y}_2$};
%\node[text width=0.5cm,  black] at (4.8, 2.0) {$\mathbf{y}_3$};
%
%\node[text width=0.5cm,  black] at (-0.6, 0) {$\mathbf{n}_4$};
%\node[text width=0.5cm,  black] at (0, 2.1) {$\mathbf{n}_5$};
%\node[text width=0.5cm,  black] at (2, -2.1) {$\mathbf{n}_6$};
%\node[text width=0.5cm,  black] at (2, 2.3) {$\mathbf{n}_7$};
\end{tikzpicture}
	\caption{An example of graphical OT.}  
	\label{fig:general_factor_clean}
\end{figure}
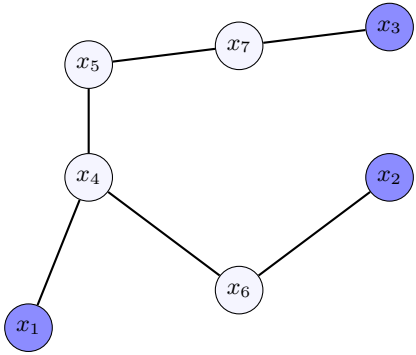
The marginal constraints $P_j(\bB)=\boldsymbol\mu_j$ for the graphical OT problem can be imposed on any variable node $j\in V$. 
Figure~\ref{fig:general_factor_clean} depicts a graph with 7 nodes. The shaded nodes in the figure correspond to marginal constraints, thus, in this example, $\Gamma= \{1, 2, 3\} \subset V =\{1,2,\ldots,7\}$. 

Consider the entropy regularized MOT problem \eqref{eq:omt_multi_regularized}. When the cost  $\bC$ has form \eqref{eq:cost_structure}, $\bK$ in \eqref{eq:K} equals $$\bK = [K(\bx)] = [\prod_{(i,j)\in E} K_{ij}(x_i,x_j)]$$ with
	\begin{equation}\label{eq:Kalpha}
		K_{ij}(x_i,x_j) = \exp (-C_{ij}(x_i,x_j)/\epsilon).
	\end{equation}
It follows that the optimal solution \eqref{eq:optB} to the entropy regularized MOT problem \eqref{eq:omt_multi_regularized} has a graphical representation as
	\begin{equation}\label{eq:MGM}
		\bB\!=\!\bK\odot\bU \!=\! [K(\bx)U(\bx)] \!=\!\! \left[\!\left(\!\prod_{(i,j)\in E} K_{ij}(x_i,x_j)\right)\!\!\left(\prod_{j\in V} u_j(x_j)\right)\right].
	\end{equation}
This is nothing but a probabilistic graphical model \cite{KolFri09}. 
%which is a new PGM whose local potentials are the Lagrangian multipliers $\bu_j$, for $j\in\Gamma$.
	
The Lagrangian approach of solving the constrained optimization problem \eqref{eq:omt_multi_regularized} seeks multipliers $\bu_j$, for $j\in\Gamma$, such that the tensor $\bB=\bK\odot\bU$ satisfies all the constraints $P_j(\bB) = \boldsymbol\mu_j$, for $j\in\Gamma$. Thus, in view of \eqref{eq:MGM}, solving the MOT problem \eqref{eq:omt_multi_regularized} is equivalent to finding a set of artificial local potentials $\bu_j$, for $j\in\Gamma$, such that the graphical model $K(\bx)U(\bx)$ in \eqref{eq:MGM} has the specified marginal distribution $\boldsymbol\mu_j$ on the $j$-th variable node for all $j\in \Gamma$.
% Note that $\bu_j =\mathbf{1}$ is a uniform potential for all $j\notin \Gamma$ and thus does not affect the graphical model $\bK\odot\bU$. 
This new perspective allows us to combine Sinkhorn algorithm and probabilistic graphical model theory to solve graphical OT problems. In particular, for fixed multipliers $\bu_1, \bu_2, \ldots, \bu_J$, calculating the projection $P_j(\bK \odot \bU)$ is exactly a Bayesian inference \cite{KolFri09} problem of inferring the $j$-th variable node over the graphical model $K(\bx)U(\bx)$.
% When $G$ does not have any loops, a condition we assume throughout, 
% Bayesian inference can be achieved efficiently using the Belief Propagation algorithm. 

When the graphical structure of $\bC$ is a tree, we arrive at the Sinkhorn belief propagation \cite{HaaSinZhaChe20} algorithm (Algorithm \ref{alg_iterative_scaling}), to solve the MOT problem \eqref{eq:omt_multi_regularized}: applying the Sinkhorn algorithm and utilizing the Belief Propagation algorithm \cite{YedFreWei03} to carry out the computation of $P_j(\bK \odot \bU)$ with the current multiplier $\bU$. Here we have assumed, without loss of generality, $\Gamma$ is a subset of the leaf nodes \cite{HaaSinZhaChe20}. Let $j_1, j_2, \ldots$ be a sequence taking values in $\Gamma$ in cyclic order and suppose the Sinkhorn algorithm is carried out in this order, then after the $k$-th iteration, $\bu_{j_k}$ is updated, and the only projection required in the next iteration is $P_{j_{k+1}}(\bK \odot \bU)$. It turns out that to evaluate $P_{j_{k+1}}(\bK \odot \bU)$, it suffices to update the messages on the path from $j_k$ to $j_{k+1}$ as used in Algorithm \ref{alg_iterative_scaling}. Compared with standard Sinkhorn algorithm, the acceleration of SBP is tremendous for MOT problems with a large number of marginals; the Belief Propagation algorithm scales well for large problem while the complexity of the brute force projection using the definition \eqref{eq:proj_discrete} grows exponentially as the number of marginals increases. 
\begin{algorithm*}[tb]
   \caption{Sinkhorn Belief Propagation (SBP) Algorithm}
   \label{alg_iterative_scaling}
\begin{algorithmic}
   \STATE Initialize the messages $m_{i\rightarrow j} (x_j)$ to be $\ett$
   \STATE Let $j_1, j_2, \ldots$ be a sequence taking values in $\Gamma$ in cyclic order
%   \STATE Update $m_{\alpha\rightarrow j} (x_j)$ and $n_{j\rightarrow \alpha}(x_j)$ using \eqref{eq:is_bp_momt1}-\eqref{eq:is_bp_momt2} until convergence
   \WHILE{not converged}
%   \FOR{$i \in\Gamma$}
        \STATE Update  $m_{j_k\rightarrow i}(x_i), i \in N(j_k)$ using 
        		\begin{subequations}\label{eq:is_bp_momt}
        		\begin{equation}
		m_{j\rightarrow i}(x_i) \propto \sum_{x_j}K_{ij}(x_i,x_j) \frac{\mu_j(x_j)}{m_{i\rightarrow j}(x_j)}, \quad \forall x_i\label{eq:is_bp_momt3}
		\end{equation}
        \STATE Update the rest of messages on the path from node $j_k$ to node $j_{k+1}$ according to 
                \begin{equation}
m_{i\rightarrow j} (x_j) \propto \sum_{x_i} K_{ij}(x_i,x_j) \prod_{k\in N(i)\backslash j}m_{k\rightarrow i}(x_i), \quad \forall x_j \label{eq:is_bp_momt1}
	\end{equation}
	\end{subequations}
    \ENDWHILE
\end{algorithmic}
\end{algorithm*}
Upon convergence of Algorithm~\ref{alg_iterative_scaling}, the solution to graphical OT problem can be obtained through $\bB=\bK\odot \bU$ with $\bU=\bu_1\otimes \bu_2\otimes \ldots \otimes\bu_J$, where $\bu_j=\boldsymbol\mu_j/\bm_{i\rightarrow j}, i\in N(j)$ for $j\in \Gamma$, and $\bu_j=\mathbf{1}$ otherwise.
%The key idea of SBP is to implement the projection $P_j(\bK\odot \bU)$ in \eqref{eq:sinkhorn_multi} using Belief Propagation.
 
% The SBP algorithm has been specialized to aggregate hidden Markov models (HMMs) whose graphical structure is shown in Figure \ref{fig:aggregateHMMs} for filtering problems for collective dynamical systems, and the resulting algorithm is termed Collective Forward-Backward algorithm \cite{SinHaaChe20}.
For more general graphical structures, we can convert it into a tree first using the junction tree algorithm \cite{KolFri09} and then apply the Sinkhorn belief propagation algorithm on the resulting junction tree. The complexity of the algorithm depends on the node size of the junction tree, which scales exponentially as the tree-width of the graph. For graphical structure with small tree width, this algorithm is still efficient. 
%	\begin{figure}[h]
%	\centering
%	\includegraphics[width=0.3\textwidth]{hmm_model}
%	\caption{Aggregate hidden Markov models}
%	\label{fig:aggregateHMMs}
%	\end{figure}

\subsection{Proximal gradient algorithm}
The proximal gradient algorithm \cite{Bec17} is an popular algorithm for the composite optimization
	\begin{equation}\label{eq:composite}
		\min_{y\in\cY} F(y) + G(y),
	\end{equation}
where $\cY$ denotes the feasibility set. The function $F$ is assumed to be smooth. The function $G$ is usually a regularizer that is possibly nonsmooth. The algorithm reads
	\begin{subequations}\label{eq:pgupdate}
	\begin{eqnarray}\label{eq:pgupdate1}
		\hspace{-0.1cm}y^{k+1} \!\!\!\!\!\!&=&\!\!\!\!\! \argmin_{y\in \cY} G(y) + \frac{1}{2 \eta} \|y-(y^k - \eta \nabla F(y^k))\|^2
		\\\!\!\!\!\!\!&=&\!\!\!\!\! \argmin_{y\in \cY} \!G(y)\! +\! \frac{1}{2 \eta} \|y\!-\!y^k\|^2\! +\! \langle \nabla F(y^k), y\!-\!y^k\!\rangle \label{eq:pgupdate2}
	\end{eqnarray}
	\end{subequations}
where $\eta>0$ is the stepsize. One advantage of the proximal gradient algorithm is that it only evaluates the gradient of $F$ and doesn't require $G$ to be differentiable. In many applications, $G$ is a regularizer of simple form, e.g., 1-norm, and the minimization \eqref{eq:pgupdate} can be implemented efficiently. 

The proximal gradient algorithm has been generalized to the non-Euclidean setting. It is built upon the mirror descent method \cite{BecTeb03,Bec17}. Let $D(\cdot, \cdot)$ be a Bregman divergence, then the generalized non-Euclidean proximal gradient algorithm reads
	\begin{equation}\label{eq:gpg}
		y^{k+1} = \argmin_{y\in \cY} G(y) + \frac{1}{\eta} D(y, y^k) + \langle \nabla F(y^k), y-y^k\rangle.
	\end{equation}
A popular choice of $D(\cdot,\cdot)$ is the Kullback-Leibler divergence ${\rm KL}(\cdot\|\cdot)$, which is suitable for optimization over probability vectors/distributions. 

The (generalized) proximal gradient algorithm has nice convergence properties. When both $F$ and $G$ are convex, the algorithm is guaranteed to converge to the global minimum with rate $\cO(1/k)$ \cite{BecTeb03,Bec17}. When $F$ is nonconvex, one can only expect for convergence to local solutions. It turns out that objective function $F(y)+G(y)$ is monotonically decreasing along the updates, and the updates converge to some stationary points with sublinear rate $\cO(1/k)$ with respect to some suitable criteria \cite{LiZhoLiaVar17}. 

\section{Density control of interacting agent systems}\label{sec:DC}
Consider a collection of dynamical systems 
	\begin{equation}\label{eq:particle}
	dX_t^i = -\frac{1}{N} \sum_{j=1}^N \nabla W(X_t^i -X_t^j) dt + u_t^idt+\sqrt{\epsilon}dB_t^i,~ i = 1,\ldots, N,
	\end{equation}
where $X_t^i\in\mR^d, u_t^i\in\mR^d$ denote the state and control of agent $i$ respectively. The disturbance is modeled by a standard Wiener process $B_t$. The $N$ agents interact with each other through an interaction potential $W$, which is assumed to be continuously differentiable and symmetric, i.e., $W(x)=W(-x), \forall x$. Clearly, $\nabla W(0) = 0$. The Hessian of $W$ is assumed to be bounded, from both above and below. We are interested in controlling the collective dynamics of the individuals \eqref{eq:particle}. Our goal is to find a common feedback strategy for the $N$ agents to steer them from an initial group configuration to a target configuration over a finite time interval $[0,\,1]$\footnote{We use the unit time interval $[0,\,1]$ to simplify the notation. A general time interval can be transformed into $[0,\,1]$ by rescaling.} with minimum effort. Let $\xi_t(x)$ be the feedback strategy of the agents, meaning $u_t^i = \xi_t(X_t^i)$. The cost function to minimize is the average quadratic control effort
	\begin{equation*}
		\mE \left\{\int_0^1 \frac{1}{2N} \sum_{i} \|\xi_t(X_t^i)\|^2 dt\right\}.
	\end{equation*}

In the mean field limit as $N\rightarrow \infty$, the group behavior can be captured by a probability distribution
	\[
		\rho_t \approx \frac{1}{N} \sum_{i=1}^N \delta_{X_t^i}
	\]
with $\delta_x$ denoting the Dirac distribution, and this density evolves according to the McKean-Vlasov equation \cite{Mck66}
	\begin{equation}\label{eq:MV}
		\partial_t\rho_t  + \nabla\cdot(\rho_t(-\nabla W*\rho_t + \xi_t)) - \frac{\epsilon}{2} \Delta \rho_t= 0.
	\end{equation}
The average control effort is approximately $\int_0^1 \int_{\mR^d} \frac{1}{2}\|\xi_t(x)\|^2 \rho_t(x) dx dt$. The initial and target configurations can both be modeled by probability distributions.
Thus, in the mean field limit, our density/distribution problem can be formulated as
	\begin{subequations}\label{eq:MFfluid}
	\begin{eqnarray}\label{eq:MFfluid1}
	\inf_{\rho, \xi} && \int_0^1 \int_{\mR^d} \frac{1}{2}\|\xi_t(x)\|^2 \rho_t(x) dx dt
	\\&&
	\partial_t\rho_t  + \nabla\cdot(\rho_t(-\nabla W*\rho_t + \xi_t))- \frac{\epsilon}{2} \Delta \rho_t = 0 \label{eq:MFfluid2}
	\\&&
	\rho_0 = \mu, \quad \rho_1 = \nu. \label{eq:MFfluid3}
	\end{eqnarray}
	\end{subequations}
One can view this as an optimal control problem for a dynamical system over the space of probability distributions with $\rho_t$ being the state. The dynamics is \eqref{eq:MFfluid2} with state $\rho_t$ and control $\xi_t$. The constraints \eqref{eq:MFfluid3} specify the initial and terminal states. We seek an optimal strategy with minimum control effort to steer the agents from an initial distribution $\mu$ to a target distribution $\nu$.
%The control cost is an action integral \eqref{eq:MFfluid1} which equals the total control effort for the collective dynamics. 

Using the Lagrangian duality method, one can derive a characterization of the solutions to \eqref{eq:MFfluid}. In particular, the optimal solution to \eqref{eq:MFfluid} can be characterized by the coupled PDEs
	\begin{subequations}\label{eq:Soptimality}
	\begin{eqnarray}\nonumber
		&& \partial_t \lambda + \frac{1}{2} \nabla \lambda ^T \nabla \lambda  -\nabla\lambda^T \nabla W*\rho_t
		\\&& - \int_{\mR^d} \rho_t(y) \nabla\lambda(y)^T \nabla W(y-x) dy +\frac{\epsilon}{2}\Delta \lambda= 0 \label{eq:Soptimality1}
		\\&&\label{eq:Soptimality2}
		\partial_t \rho_t +\nabla\cdot(\rho_t (-\nabla W*\rho_t  + \nabla\lambda)) - \frac{\epsilon}{2} \Delta\rho_t = 0
		\\&&
		\rho_{0} = \mu, \quad \rho_{1} = \nu \label{eq:Soptimality3},
	\end{eqnarray}
	\end{subequations}	
where $\lambda$ is the Lagrange multiplier associated with the continuity constraint \eqref{eq:MFfluid2}. The optimal control policy is a state feedback 
	\begin{equation*}
		\xi_t(x) = \nabla \lambda(t,x).
	\end{equation*}

There are several potential approaches to compute an optimal solution to the density control problem \eqref{eq:MFfluid}. For instance, the optimality condition \eqref{eq:Soptimality} can be viewed as the Pontryagin's principle for \eqref{eq:MFfluid} when \eqref{eq:MFfluid} is treated as an optimal control problem with state $\rho_t$ \cite{CheGeoPav21b}. The multiplier $\lambda$ then becomes the costate in the Pontryagin's principle \cite{LeeMar67,FleRis75}. To get a solution to \eqref{eq:Soptimality}, one can use indirect method such as shooting method \cite{LeeMar67,FleRis75} that is widely adopted for optimal control problems. However, due to the coupling between the state $\rho_t$ and the costate $\lambda$, and more importantly the fact that they are of infinite dimension, the shooting method maybe unstable and is not guaranteed to converge. Next we present a completely different approach to solve \eqref{eq:MFfluid} based on a reformulation. 
	
\subsection{Reformulation and Discretization}
For a given feedback policy $\xi_t$, in the mean field limit, the distribution $\rho_t$ of the individuals follows the McKean-Vlasov equation \eqref{eq:MV} and is deterministic. Moreover, the interaction between agents is of the form $-\frac{1}{N} \sum_{j=1}^N \nabla W(X_t^i -X_t^j)\approx -\nabla W*\rho_t$, which only depends on the group behavior. Thus, when $N$ is sufficiently large, the interactions between an agent with other agents becomes the interaction between the agent and the deterministic group distribution $\rho_t$. By the theory of propagation of chaos \cite{MelRoe87}, the $N$ agents become effectively independent to each other and each of them follows the same stochastic dynamics 
	\begin{equation}\label{eq:MFdynamics}
		dX_t = -[\nabla W*\rho_t](X_t) dt + \xi_t(X_t)dt+\sqrt{\epsilon} dB_t.
	\end{equation}

Denote by $\cP$ the distribution induced by \eqref{eq:MFdynamics} over the path space $\Omega=C([0,1],\mR^d)$, and by $\cQ(\cP)$ be distribution induced by the process
	\begin{equation}\label{eq:MFdynamicszero}
		dX_t = -\nabla W*\rho_t dt +\sqrt{\epsilon} dB_t,
	\end{equation}
then by the Girsanov theorem \cite{FleRis75,KarShr88}, following a similar argument as in the Schr\"odinger bridge problem \eqref{eq:SBmeasure}-\eqref{eq:daipra}, we obtain 
	\[
		{\rm KL}(\cP \| \cQ(\cP)) = \int_0^1 \int_{\mR^d} \frac{1}{2\epsilon}\|\xi_t(x)\|^2 \rho_t(x) dx dt.
	\]
Note that we used $\cQ(\cP)$ to emphasize the fact that $\cQ$ depends on the marginal flow of $\cP$, denoted by $(X_t)_\sharp \cP = \rho_t$. 

Consequently, the density control problem \eqref{eq:MFfluid} can be reformulated as
	\begin{subequations}\label{eq:MFSB}
	\begin{eqnarray}
		\min_{\cP} && {\rm KL}(\cP \| \cQ(\cP))
		\\
		&& (X_0)_\sharp \cP = \mu, \quad (X_1)_\sharp \cP = \nu.
	\end{eqnarray}
	\end{subequations} 	
This formulation \eqref{eq:MFSB} coincides with the mean field Schr\"odinger bridge problem \cite{BacConGenLeo20}. The equivalence between \eqref{eq:MFfluid} and \eqref{eq:MFSB} is rigorously justified in \cite{BacConGenLeo20}, extending the large deviation theory to interacting particle systems. The major difference between \eqref{eq:MFSB} and the standard Schr\"odinger bridge problem \eqref{eq:SBmeasure} lies in the fact that the prior distribution $\cQ$ in the former depends on the solution $\cP$, rendering a nonconvex optimization, in general, over the space of path distributions.

%This problem formulation coincides with the mean-field Schr\"odinger bridge problem \cite{BacConGenLeo20}. It can be equivalently formulated as an optimization over the space of path distributions\footnote{the distributions over the path space $\Omega=C([0,1],\mR^n)$.} as 
%	\begin{subequations}\label{eq:MFSB}
%	\begin{eqnarray}
%		\min_{\cP} && {\rm KL}(\cP \| \cQ(\cP))
%		\\
%		&& (X_0)_\sharp \cP = \mu, \quad (X_1)_\sharp \cP = \nu,
%	\end{eqnarray}
%	\end{subequations} 
%where $\cQ(\cP)$ is the distribution over the path space associated with the diffusion process 
%	\begin{equation}\label{eq:MFdynamics}
%		dX_t = -\nabla W*\rho_t dt + \sqrt{\epsilon} dB_t,
%	\end{equation}
%and $\rho_t = (X_t)_\sharp \cP$ is the marginal distribution of $\cP$. The major difference between \eqref{eq:MFSB} and the standard Schr\"odinger bridge problem lies in the fact that the prior distribution $\cQ$ in the former depends on the solution $\cP$. 
%
%The equivalence between \eqref{eq:MFfluid} and \eqref{eq:MFSB} is rigorously justified in \cite{BacConGenLeo20} through the lens of large deviation theory. 
%When $N$ is large, by propagation of chaos \cite{??}, the $N$ agents become effectively independent to each other and each of them follows the same stochastic dynamics \eqref{eq:MFdynamics}. The equivalence between ${\rm KL}(\cP \| \cQ(\cP))$ and $\int_0^1 \int_{\mR^d} \frac{1}{2}|\xi_t|^2 \rho_t(x) dx dt$ is then a consequence of the Girsanov theorem \cite{??} as in \cite{??}. 

The optimization variable $\cP$ of \eqref{eq:MFSB} is of infinite dimension. To develop an implementable algorithm for \eqref{eq:MFSB}, we first discretize the problem in time $t_i = i/T,\,i=0,1,\ldots,T$ as well as in space over a grid. With this discretization, the path distribution $\cP$ becomes a $(T+1)$-dimensional tensor $\bM$ with $M(x_0,x_1,\ldots,x_T)$ representing the probability of the process $\cP$ goes through a neighborhood of $(X_{0} = x_0, X_{1/T} = x_1,\ldots, X_1 =x_T)$. In terms of $\bM$, the objective function ${\rm KL}(\cP \| \cQ(\cP))$ becomes
	 \[
	 	\langle \bC(\bM), \bM\rangle + \epsilon\langle \bM, \log \bM\rangle
	\]
where
	\[
		\langle \bM, \log \bM\rangle\! =\!\!\! \sum_{x_0,x_1,\ldots,x_T} M(x_0,x_1,\ldots, x_T) \log M(x_0,x_1,\ldots, x_T),
	\]
and $C(\bM)(x_0,x_1,\cdots,x_T)$ is the minimum control effort to drive the deterministic version ($\epsilon=0$) of \eqref{eq:MFdynamics} to go through the state $x_0,x_1,\cdots,x_T$. More explicitly, when the discretization grid is sufficiently fine, 
	\begin{equation*}
		C(\bM)(x_0,x_1,\cdots,x_T) \!=\!\! \sum_{i=0}^{T-1}\frac{T}{2} \|x_{i+1}-x_i+\frac{1}{T} [\nabla W * P_i(\bM)](x_i)\|^2,
	\end{equation*}
where $P_i(\bM)$ denotes the marginal of $\bM$ over $x_i$ and, by abuse of notation, $\nabla W*P_i(\bM)$ is a discretization of the convolution. 

Thus, after discretization, \eqref{eq:MFSB} becomes
	\begin{subequations}\label{eq:MOTform}
	\begin{eqnarray}
	\min_{\bM} && \langle \bC(\bM), \bM\rangle + \epsilon\langle \bM, \log \bM\rangle
	\\&&
	P_0(\bM) = \boldsymbol\mu, \quad P_T(\bM) = \boldsymbol\nu.
	\end{eqnarray}
	\end{subequations}
Here $\boldsymbol\mu$ and $\boldsymbol\nu$ denote the discretized version of $\mu$ and $\nu$ respectively. This formulation \eqref{eq:MOTform} is akin to the MOT problem except that the unit transport cost tensor $\bC$ now depends on the optimization variable $\bM$. This difference excludes the possibility of applying the Sinkhorn type algorithm directly to solve \eqref{eq:MOTform}. Next we develop an algorithm to compute the solution to \eqref{eq:MOTform} by sequentially linearizing $\langle \bC(\bM), \bM\rangle$ and then solving the resulting MOT problems.

\subsection{Proximal Sinkhorn Belief Propagation Algorithm}	
%The major difficulty to solve \eqref{eq:MOTform} lies in the dependency of $C$ over $M$, which deviates \eqref{eq:MOTform} from standard entropy regularized multi-marginal optimal transport problem. One strategy to solve \eqref{eq:MOTform} is (generalized) proximal gradient algorithm \cite{ParBoy14}. 
Denote $\Pi(\boldsymbol\mu,\boldsymbol\nu)$ the set of $\bM$ that is consistent with the marginals $\boldsymbol\mu, \boldsymbol\nu$ and $F(\bM) = \langle \bC(\bM), \bM\rangle$ then \eqref{eq:MOTform} reads
	\begin{equation}\label{eq:composite}
		\min_{\bM\in \Pi(\boldsymbol\mu,\boldsymbol\nu)} F(\bM) -\epsilon \cH(\bM).
	\end{equation}
This is a composite optimization over the probability simplex. We can thus apply the generalized proximal gradient descent algorithm to solve it. Surprisingly, when the Bregman divergence in \eqref{eq:gpg} is chosen to be the Kullback-Leibler divergence, each iteration of the algorithm on the problem \eqref{eq:composite} takes the form
	\begin{equation}\label{eq:proximal}
		\bM_{k+1}\! =\! \argmin_{\bM\in \Pi(\boldsymbol\mu,\boldsymbol\nu)}\! \langle \nabla F(\bM_k), \bM\rangle \!+\frac{1}{\eta} {\rm KL}(\bM\| \bM_k) \!- \epsilon \cH(\bM) 
	\end{equation}
where $\eta>0$ is the step size. Expanding the KL divergence term, the above becomes
		\begin{equation}\label{eq:proximalMOT}
		\bM_{k+1}\! =\! \argmin_{\bM\in \Pi(\boldsymbol\mu,\boldsymbol\nu)}\! \langle \nabla F(\bM_k)\!-\frac{1}{\eta}\!\log \bM_k,\! \bM\rangle\! -\! (\epsilon+\frac{1}{\eta}) \cH(\bM)
	\end{equation}
which is a standard entropy regularized multi-marginal optimal transport problem \eqref{eq:omt_multi_regularized} with cost tensor $\nabla F(\bM_k)-\frac{1}{\eta}\log \bM_k$.
\begin{prop}\label{prop:nablaF}
The gradient of $F(\bM) = \langle \bC(\bM), \bM\rangle$ is
	\begin{equation}\label{eq:nablaf}
		\nabla F(\bM) = \bC(\bM) + E(\bM)
	\end{equation}
where $E(\bM)(x_0,x_1,\ldots,x_T) = \sum_{i=0}^{T-1} E_i(x_i)$ with
	\begin{eqnarray}\nonumber
		E_{i} (y) &=& \sum_{x_i,x_{i+1}}\nabla W(x_i-y)^T[x_{i+1}-x_i+\frac{1}{T} \nabla W * P_i(\bM)]
		\\&&P_{i,i+1}(\bM)(x_i,x_{i+1}). \label{eq:Ei}
	\end{eqnarray}
\end{prop}
%Note that for the choice $D(M,M_k) = {\rm KL} (M, M_k)$, each iteration \eqref{eq:proximal} is equivalent to solving a standard entropy regularized multi-marginal optimal transport problem, for which the Sinkhorn (iterative scaling) belief propagation \cite{HaaSinZhaChe20,SinHaaChe20} is an efficient algorithm. By iteratively updating $M$ as \eqref{eq:proximal} we arrive at an efficient algorithm for \eqref{eq:MOTform}. Even though \eqref{eq:MOTform} is not a convex problem in general, the proximal gradient algorithm \eqref{eq:proximal} do converge to local minimum in sublinear rate. 
\begin{proof}
%To apply the proximal gradient algorithm to the problem \eqref{eq:MOTform}, we need to evaluate $\nabla f(M)$ with $f(M) = \langle C(M), M\rangle$. 
By definition, 
	\begin{equation*}
		F(\bM+\delta \bM) - F(\bM) \approx \langle \nabla F(\bM), \delta \bM\rangle.
	\end{equation*}
It follows that
	\begin{eqnarray*}
		&&\langle \nabla F(\bM), \delta \bM\rangle = \langle \bC(\bM), \delta \bM\rangle
		\\&&\hspace*{-0.9cm}+ \sum_{i=0}^{T-1} T\langle  \frac{1}{T} (\nabla W * P_i(\delta \bM))^T(x_{i+1}-x_i\!+\!\frac{1}{T} \nabla W * P_i(\bM)),\bM\rangle. 
	\end{eqnarray*}
%	\begin{equation}
%		\langle \nabla F(\bM), \delta \bM\rangle = \langle \bC(M), \delta M\rangle+ \sum_{i=0}^{T-1} T\langle  \frac{1}{T} \nabla W * ((X_i)_\sharp \delta M))^T(x_{i+1}-x_i+\frac{1}{T} \nabla W * ((X_i)_\sharp M)),M\rangle. 
%	\end{equation}
The second term on the right hand side equals
	\begin{eqnarray*}
		&&\hspace*{-1cm} \sum_{i=0}^{T-1} \langle  (\nabla W * P_i(\delta \bM))^T(x_{i+1}\!-x_i\!+\frac{1}{T} \nabla W * P_i(\bM)),P_{i,i+1}(\bM)\rangle
		\\&=& \sum_{i=0}^{T-1}\langle  \delta \bM, E_{i} (x_i) \rangle
	\end{eqnarray*}
where $E_i$ is as in \eqref{eq:Ei}.
%	\begin{eqnarray*}
%		E_{i} (y) &=& \sum_{x_i,x_{i+1}}\nabla W(x_i-y)^T[x_{i+1}-x_i+\frac{1}{T} \nabla W * P_i(\bM)]
%		\\&&P_{i,i+1}(\bM)(x_i,x_{i+1}).
%	\end{eqnarray*}
Hence,
	\begin{equation*}
		\langle \nabla F(\bM), \delta \bM\rangle = \langle \bC(\bM), \delta \bM\rangle + \langle E(\bM), \delta \bM\rangle,
	\end{equation*}
with $E(\bM)(x_0,x_1,\ldots,x_T) = \sum_{i=0}^{T-1} E_i(x_i)$, and therefore
	 \begin{equation*}
	 	\nabla F(\bM) = \bC(\bM) + E(\bM).
	 \end{equation*}
This completes the proof.	
\end{proof}

Plugging \eqref{eq:nablaf} into \eqref{eq:proximalMOT} yields the proximal gradient iteration
%	\begin{equation}\label{eq:proximalSB}
%		M_{k+1} = \argmin_{M\in \Pi(\mu,\nu)} \langle  C(M_k)+F(M_k), M\rangle + \frac{1}{\eta} {\rm KL} (M, M_k) - \epsilon H(M). 
%	\end{equation}
%Each iteration requires solving the standard Schr\"odinger bridge problem
	\begin{eqnarray}\nonumber
		\bM_{k+1} &=& \argmin_{\bM\in \Pi(\boldsymbol\mu,\boldsymbol\nu)} \langle \bC(\bM_k)+E(\bM_k)-\frac{1}{\eta} \log \bM_k, \bM\rangle 
		\\&& -(\epsilon + \frac{1}{\eta})\cH(\bM).\label{eq:proximaliter}
	\end{eqnarray}
Note that both $\bC(\bM_k)$ and $E(\bM_k)$ have a graphical structure associated with the line graph (Figure \ref{fig:graphical}). Thus, assuming $\bM_k$ has the same graphical structure, the solution $\bM_{k+1}$ to \eqref{eq:proximaliter} also has a graphical structure corresponding to the line graph. Therefore, with proper initialization, each iteration \eqref{eq:proximaliter} can be solved efficiently using the Sinkhorn Belief Propagation algorithm (Algorithm \ref{alg_iterative_scaling}). We thus establish our Proximal Sinkhorn Belief Propagation algorithm (Algorithm \ref{alg:proximalsinkhorn}) to solve \eqref{eq:MOTform}.
\begin{algorithm*}[tb]
   \caption{Proximal Sinkhorn Belief Propagation algorithm}
   \label{alg:proximalsinkhorn}
\begin{algorithmic}
   \STATE Input: cost tensor $\bC$, regularization $\epsilon$, stepsize $\eta$, number of iterations $K$
   \STATE Initialize $\bM_1$ to be a uniform probability vector
%   \WHILE{not converged}
   \FOR{$k = 1, 2, 3, 
   \ldots,K$}
        \STATE Compute $\bC(\bM_k)+E(\bM_k)-\frac{1}{\eta} \log \bM_k$
        \STATE Solve \eqref{eq:proximaliter} using the Sinkhorn Belief Propagation algorithm to obtain $\bM_{k+1}$
    \ENDFOR
%    \ENDWHILE
\end{algorithmic}
\end{algorithm*}
	\begin{figure}[h]
	\centering
	\begin{tikzpicture}[scale=0.95, every node/.style={scale=0.95}]
     
       \tikzstyle{bluenode}=[draw, circle, minimum size=9mm, inner sep=1pt, fill=blue!25];
       \tikzstyle{bluenode1}=[draw, circle, minimum size=9mm, inner sep=1pt, fill=blue!1];
        % Change vertex color and intensity here
  \tikzstyle{edge}=[>=stealth, thick, black] % Change edge color and intensity here
  
  \tikzstyle{ref}=[draw, circle, minimum size=0.2em, inner sep=1pt, color = white, fill=white];

\node [bluenode] at (2,0) (b1) {$x_0$};
\node [bluenode1] at (4,0) (b2) {$x_1$};
\node [bluenode1] at (6,0) (b3) {$x_{T-1}$};
\node [bluenode] at (8,0) (b4) {$x_T$};

\draw[edge] (b1) to (b2);
%\draw[edge] (b2) to (b3);
\draw[loosely dotted, very thick] (b2) -- (b3); 
\draw[edge] (b3) to (b4);
\end{tikzpicture}
	\caption{Graph for the graphical OT \eqref{eq:proximaliter}}
	\label{fig:graphical}
	\end{figure}
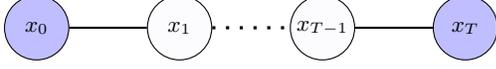
\begin{remark}
The Proximal Sinkhorn Belief Propagation algorithm inherits the convergent properties of proximal gradient algorithm and converges to a solution with sublinear rate $\cO(1/k)$. Note that the problem \eqref{eq:MOTform} is in general non-convex and thus the convergence is to a local solution. Each iteration of our algorithm requires solving a graphical OT problem using the Sinkhorn Belief Propagation algorithm. Let $D$ be the number of discretized grid points over space, then the complexity of the Sinkhorn Belief Propagation is $\cO(D^2 T)$. 
\end{remark}
\begin{remark}
In the limit case where $\epsilon=0$, the stochastic disturbance in the dynamics vanishes and agents become deterministic. Note that Algorithm \ref{alg:proximalsinkhorn} applies to this deterministic setting.
\end{remark}
	
%\subsection{Accelerated Bregman proximal gradient}
%no need. save time
	
\subsection{Optimal control strategy}
The optimal control policy is $\xi_t(x) = \nabla \lambda(t,x)$.
Once Algorithm \ref{alg:proximalsinkhorn} converges, the corresponding control policy can be recovered by solving the linear equation
	\[
		\partial_t\rho_t + \nabla\cdot(\rho_t(-\nabla W*\rho_t + \nabla \lambda(t,\cdot))) - \frac{\epsilon}{2} \Delta \rho_t = 0.
	\]
More specifically, 
	\[
		\partial_t\rho_t  + \nabla\cdot(\rho_t(-\nabla W*\rho_t))- \frac{\epsilon}{2} \Delta \rho_t
	\]
can be estimated using the solution $\bM^\star$ to \eqref{eq:MOTform}, denoted by $g_t$. It follows that $\lambda$ can be recovered by solving the linear equation
	\[
		 \nabla\cdot(\rho_t \nabla\lambda(t,\cdot)) = -g_t,
	\]
or more precisely the least square problem
	\[
		\min_{\lambda} \|\nabla\cdot(\rho_t \nabla\lambda(t,\cdot))+g_t\|^2.
	\]

An alternative approach is based on the fact that the optimal $\cP$ is associated with the stochastic process
	\[
		dX_t = -\nabla W*\cP_t dt + \nabla\lambda(t,X_t)dt+\sqrt{\epsilon} dB_t.
	\]
The joint distribution of $X_{i/T}$ and $X_{(i+1)/T}$ of this process is approximately
	\begin{eqnarray*}
		&&\cP_{i/T}(X_{i/T}) \cN (X_{(i+1)/T}; X_{i/T}
		\\&&+[-(\nabla W*\cP_{i/T})(X_{i/T})+\nabla\lambda(i/T,X_{i/T})]/T,\epsilon/T).
	\end{eqnarray*}
On the other hand, it is approximated by $P_{i,i+1}(\bM^\star)$. Combining these two expressions we can solve $\lambda$ and thus the optimal control policy. 
%A better way to recover the control is through the dual PDE for $\lambda$; see \cite{BacConGenLeo20}.

\subsection{Extension to general dynamics and cost}
In the above discussions, to better illustrate our density control framework for interacting agent systems, we have restricted our attention to the simple dynamics \eqref{eq:particle}. Now we extend this framework to more general dynamics\footnote{The dependence of $b, \sigma$ over time is suppressed to simplify the notation.}
	\begin{eqnarray}\nonumber
		dX_t^i \!\!&=&\!\! -\frac{1}{N} \sum_{j=1}^N \nabla W(X_t^i \!-\!X_t^j) dt \!+\! b(X_t)dt \!+\! \sigma (u_t^idt+\sqrt{\epsilon}dB_t^i)
		\\&& i = 1,\ldots, N,\label{eq:particlegeneral}
	\end{eqnarray}
where $b(\cdot)\in\mR^d$ is a continuous drift term and $\sigma \in \mR^{d\times p}$ is the input matrix,
%We assume $a = \sigma\sigma'$ is nonsingular. 
and more general cost function
	\begin{equation*}
		\int_0^1 \int_{\mR^d} [\frac{1}{2}\|\xi_t(x)\|^2 +V(x)]\rho_t(x) dx dt.
	\end{equation*}
In the mean field limit, the density control problem can be formulated as 
	\begin{subequations}\label{eq:Gfluid}
	\begin{eqnarray}\label{eq:Gfluid1}
	\inf_{\rho, \xi} &&\!\!\!\!\!\!\!\!\!\! \int_0^1 \int_{\mR^d} [\frac{1}{2}\|\xi_{t}(x)\|^2 + V(x)] \rho_{t}(x) dx dt
	\\&&\!\!\!\!\!\!\!\!\!\!\!\!\!\!\!
	\partial_t\rho_{t} \!+\! \nabla\!\cdot\!(\rho_{t}(\!-\!\nabla W\!*\!\rho_{t} \!+\! b\!+\!\sigma \xi_{t})) \!\!- \!\frac{\epsilon}{2}\!\! \sum_{i,k}\! \frac{\partial^2 (a_{ik}\rho_{t})}{\partial x_i \partial x_k}\! =\! 0 \label{eq:Gfluid2}
	\\&&\!\!\!\!\!\!\!\!\!\!\!\!
	\rho_{0} = \mu, \quad \rho_{1} = \nu \label{eq:Gfluid3}
	\end{eqnarray}
	\end{subequations}
where $a = \sigma\sigma^T$. 
	
The optimal strategy of \eqref{eq:Gfluid} is 
	\begin{equation*}
		\xi_t(x) = \sigma^T\nabla\lambda(t,x)
	\end{equation*}
where $\lambda$ solves the PDEs
	\begin{subequations}\label{eq:Goptimality}
	\begin{eqnarray}\nonumber
		&& \partial_t \lambda + \frac{1}{2} \nabla \lambda ^T a \nabla \lambda -V+\nabla\lambda^T b-\nabla\lambda^T \nabla W*\rho
		\\&&\hspace{-0.9cm} - \int \rho(y) \nabla\lambda(y)^T \nabla W(y-x) dy +\frac{\epsilon}{2}\tr(a\nabla^2\lambda) = 0 \label{eq:Goptimality1}
		\\&&\label{eq:Goptimality2}
		\hspace{-0.9cm}\partial_t \rho_t\!\! +\!\nabla\!\cdot\!(\rho_t (\!-\nabla W*\rho_t \! +\! b\!+\!a\nabla\lambda)) \!-\! \frac{\epsilon}{2} \sum_{i,k} \frac{\partial^2 (a_{ik}\rho_{t})}{\partial x_i \partial x_k} \!=\! 0
		\\&&\hspace{-0.9cm}
		\rho_{0} = \mu, \quad \rho_{1} = \nu \label{eq:Goptimality3}.
	\end{eqnarray}
	\end{subequations}	
	
Following similar arguments as before we obtain an alternative formulation
	\begin{subequations}\label{eq:GSB}
	\begin{eqnarray}
		\min_{\cP} && \epsilon{\rm KL}(\cP \| \cQ(\cP)) + \int V d\cP
		\\
		&& (X_0)_\sharp \cP = \mu, \quad (X_1)_\sharp \cP = \nu,
	\end{eqnarray}
	\end{subequations} 
where $\cQ(\cP)$ is the distribution over the path space associated with the diffusion process 
	\begin{equation}\label{eq:Gdynamics}
		dX_t = -\nabla W*\cP_t dt + b(X_t) dt + \sqrt{\epsilon}\sigma dB_t.
	\end{equation}
The same as \eqref{eq:MOTform}, after discretization over space and time, the problem can be written as  
	\begin{equation}\label{eq:MOTgeneral}
	\min_{\bM\in\Pi(\boldsymbol\mu,\boldsymbol\nu)}  \langle \bC(\bM), \bM\rangle + \epsilon\langle \bM, \log \bM\rangle,
	\end{equation}
but with a slightly different cost tensor
	\begin{eqnarray}\nonumber
		&&C(\bM)(x_0,x_1,\cdots,x_T) = \frac{1}{T}\sum_{i=0}^{T-1}V(x_i)
		\\&&+\sum_{i=0}^{T-1}\frac{T}{2} \|x_{i+1}-x_i+\frac{1}{T} \nabla W * P_i(\bM)-\frac{1}{T} b(x_i)\|^2.\label{eq:CV}
	\end{eqnarray}

Let $F(\bM) = \langle \bC(\bM), \bM\rangle$, then the above becomes a composite optimization \eqref{eq:composite} and can be solved using the proximal gradient algorithm. The derivation is similar to that of Proposition \ref{prop:nablaF} and is omitted. 
\begin{prop}\label{prop:nablaF2}
The gradient of $F(\bM) = \langle \bC(\bM), \bM\rangle$ with $\bC$ in \eqref{eq:CV} is
	\begin{equation*}
		\nabla F(\bM) = \bC(\bM) + E(\bM)
	\end{equation*}
where $E(\bM)(x_0,x_1,\ldots,x_T) = \sum_{i=0}^{T-1} E_i(x_i)$ with
	\begin{eqnarray*}
		E_{i} (y) &=& \sum_{x_i,x_{i+1}}\nabla W(x_i-y)^T[x_{i+1}-x_i+\frac{1}{T} \nabla W * P_i(\bM)
		\\&&-\frac{1}{T} b(x_i)]P_{i,i+1}(\bM)(x_i,x_{i+1}).
	\end{eqnarray*}
\end{prop}
%\begin{proof}
%may be skipped
%\end{proof}
The proximal Sinkhorn belief propagation algorithm can be applied directly to solve \eqref{eq:MOTgeneral} with a small modification on the expression of $E(\bM)$ as in Proposition \ref{prop:nablaF2}. 

\section{Density control with multiple species}\label{sec:multi}
In this section, we extend our density control framework to account for the collective dynamics with multiple species. Consider a group of individuals comprised of $L$ species and each has $N$ agents. The dynamics of the $i$-th agents in the $\ell$-th species is
	\begin{eqnarray*}
		dX_{\ell,t}^i &=& -\frac{1}{N} \sum_{m=1}^L \sum_{j=1}^N \nabla W_{\ell m}(X_{\ell,t}^i -X_{m,t}^j) dt + b_\ell(X_{\ell,t}^i)dt
		\\&& + \sigma(u_{\ell,t}^idt+\sqrt{\epsilon} dB_{\ell,t}^i),\quad i = 1,\ldots, N,
	\end{eqnarray*}
where $X_{\ell,t}^i$ and $u_{\ell,t}^i$ denote the state and control of the $i$-th agents in the $\ell$-th species respectively. The interaction potential between species $\ell$ and species $m$ is assumed to be continuously differentiable and symmetric in the sense $W_{\ell m} (x) = W_{\ell m}(-x) =W_{m\ell} (x) = W_{m\ell}(-x)$. We seek $L$ feedback policies, one for each species, such that, when they are adopted by the individuals, the group would be transformed from one configuration to another.  

In the mean field region, denoting the initial distribution/configuration of species $\ell$ by $\mu_\ell$, and its target distribution by $\nu_\ell$, this density control problem can be formulated as 
	\begin{subequations}\label{eq:MSfluid}
	\begin{eqnarray}\label{eq:MSfluid1}
	\inf_{\rho_1,\ldots,\rho_L, \xi_1,\ldots,\xi_L} &&\hspace{-0.7cm} \sum_{\ell=1}^L\int_0^1 \int_{\mR^d} [\frac{1}{2}\|\xi_{\ell,t}(x)\|^2 + V_\ell (x)] \rho_{\ell,t}(x) dx dt
	\\&&\hspace{-0.7cm} \nonumber
	\partial_t\rho_{\ell,t}\! +\! \nabla\!\cdot\!(\!\rho_{\ell,t}(\!-\!\sum_{m} \nabla W_{\ell m}*\rho_{m,t} + b_\ell+\sigma\xi_{\ell,t})) 
	\\&&\hspace{-0.5cm} - \frac{\epsilon}{2} \sum_{i,k} \frac{\partial^2 (a_{ik}\rho_{\ell,t})}{\partial x_i \partial x_k} = 0,~\ell = 1,2,\ldots, L \label{eq:MSfluid2}
	\\&&\hspace{-0.7cm} 
	\rho_{\ell,0} = \mu_\ell, \quad \rho_{\ell,1} = \nu_\ell, ~\ell = 1,2,\ldots, L, \label{eq:MSfluid3}
	\end{eqnarray}
	\end{subequations}
where $\rho_{\ell,t}$ denotes the distribution of the $\ell$-th species. 
The constraints \eqref{eq:MSfluid2} is a generalization of \eqref{eq:Gfluid2} to the multi-species setting describing the distribution evolution of each species. 	
The optimal solution to \eqref{eq:MSfluid} can be characterized by the PDEs
	\begin{subequations}\label{eq:optimality}
	\begin{eqnarray}\label{eq:optimality1}
		&&\hspace{-0.5cm} \partial_t \lambda_\ell + \frac{1}{2} \nabla \lambda_\ell ^T a \nabla \lambda_\ell \!-\!V_\ell +\nabla \lambda_\ell^T b_\ell \!-\!\nabla\lambda_\ell^T \sum_m \nabla W_{\ell m}\! *\!\rho_{m,t} 
		\\&&\hspace{-0.5cm}\nonumber\!-\! \sum_m \int \rho_{m,t}(y) \nabla\lambda_m(y)^T \nabla W_{m\ell} (y-x) dy \!+\!\frac{\epsilon}{2}\tr(a\nabla^2 \lambda_\ell) = 0
		\\\nonumber&&\hspace{-0.5cm}
		\partial_t \rho_{\ell,t}  +\nabla\cdot(\rho_{\ell,t} (-\sum_m \nabla W_{\ell m}*\rho_m +b_\ell + a \nabla\lambda_\ell))
		\\&&\label{eq:optimality2} - \frac{\epsilon}{2} \sum_{i,k} \frac{\partial^2 (a_{ik}\rho_{\ell,t})}{\partial x_i \partial x_k}= 0
		\\&&\hspace{-0.5cm}
		\rho_{\ell,0} = \mu_\ell, \quad \rho_{\ell,1} = \nu_\ell \label{eq:optimality3}
	\end{eqnarray}
	\end{subequations}	
for all $\ell = 1,2,\ldots, L$. Here $\lambda_1,\lambda_2,\ldots,\lambda_L$ are Lagrange multipliers associated with the constraints \eqref{eq:MSfluid2} for $\ell = 1,2,\ldots, L$. The corresponding optimal control policy for the $\ell$-th species is 
	\begin{equation*}
		u_{\ell,t} = \sigma^T \nabla \lambda_\ell(t, X_{\ell,t}). 
	\end{equation*}

To develop an efficient algorithm for \eqref{eq:MSfluid}, we reformulate it as an optimization over the path measures. More specifically, denote by $\cP^\ell$ the distribution on the path space induced by species $\ell$, then following a similar argument as before, we obtain the following reformulation 
	\begin{subequations}\label{eq:MSSB}
	\begin{eqnarray}
		\hspace{-0.1cm}\min_{\cP^1,\ldots,\cP^L} &&\!\!\!\!\!\!\!\! \sum_{\ell=1}^L \left\{\epsilon{\rm KL}(\cP^\ell \| \cQ^\ell(\cP^1,\cdots,\cP^L))+\int V_\ell d\cP^\ell\right\}
		\\
		&&\!\!\!\!\!\!\!\! (X_0)_\sharp \cP^\ell = \mu_\ell, (X_1)_\sharp \cP^\ell = \nu_\ell,\ell = 1,2,\ldots, L.
	\end{eqnarray}
	\end{subequations} 
Here the distribution $\cQ^\ell$ is induced by the diffusion process
	\begin{equation*}
		dX_t = -\sum_m [\nabla W_{\ell m}*\cP_{t}^m](X_t) dt+ b_\ell(X_t)dt + \sqrt{\epsilon}\sigma dB_t,
	\end{equation*}
which clearly depends on $\cP^1,\cP^2,\cdots,\cP^L$.

Discretizing the problem over space and time, the optimization variables become $L$ tensors $\bM^\ell, \ell = 1,2,\ldots, L$ and the optimization problem becomes
	\begin{subequations}\label{eq:LM}
	\begin{eqnarray}
	\hspace{-0.4cm}\min_{\bM^1,\cdots,\bM^L} &&\hspace{-0.7cm} \sum_{\ell=1}^L \{\langle \bC^\ell(\bM^1,\ldots,\bM^L), \bM^\ell\rangle + \epsilon\langle \bM^\ell, \log \bM^\ell\rangle\}
	\\&&\hspace{-0.7cm}
	P_0(\bM^\ell) = \mu_\ell, \quad P_T(\bM^\ell) = \nu_\ell,  ~\ell = 1,2,\ldots, L,
	\end{eqnarray}
	\end{subequations}
where the cost tensor for the $\ell$-th species is
	\begin{eqnarray*}
		&&C^\ell(\bM^1,\ldots,\bM^L)(x_0,x_1,\cdots,x_T) = \sum_{i=0}^{T-1}\frac{1}{T}V_\ell(x_i)
		\\&&\hspace{-0.3cm}+\sum_{i=0}^{T-1}\frac{T}{2} \|x_{i+1}-x_i+\frac{1}{T} \sum_m\nabla W_{\ell m} * P_i(\bM^m)-\frac{1}{T} b_\ell(x_i)\|^2.
	\end{eqnarray*}

A more compact form of the above problem can be obtained by combining the $L$ optimization variables $\bM^1,\bM^2,\cdots,\bM^L$ into a single variable $\bM$. More precisely, we denote by $\bM$ the $T+2$ dimensional tensor where the index for the first dimension is $\ell$. Similarly, we combine $\bC^1,\bC^2,\ldots,\bC^L$ into $\bC$ where
	\begin{eqnarray}\label{eq:Cmulti}
		&&C(\bM)(\ell,x_0,x_1,\cdots,x_T) = \sum_{i=0}^{T-1}\frac{1}{T}V_\ell(x_i)+
		\\&&\hspace{-0.7cm}\sum_{i=0}^{T-1}\frac{T}{2} \|x_{i+1}\!-\!x_i\!+\!\frac{1}{T} \sum_m[\nabla W_{\ell m} *P_{-1,i}(\bM)](\ell, x_i)\!-\!\frac{1}{T} b_\ell(x_i)\|^2.\nonumber
	\end{eqnarray}
In the above, we adopt an unconventional notation $P_{-1,i}(\bM)$ to denote the marginal of $\bM$ over $(\ell, x_i)$. In terms of $\bM, \bC$, the above optimization \eqref{eq:LM} can be rewritten as
	\begin{subequations}\label{eq:MOTmulti}
	\begin{eqnarray}\label{eq:MOTmulti1}
	\min_{\bM} && \langle \bC(\bM), \bM\rangle + \epsilon\langle \bM, \log \bM\rangle
	\\&&\label{eq:MOTmulti2}
	P_{-1,0}(\bM)= \boldsymbol\mu, \quad P_{-1,T}(\bM) =\boldsymbol\nu
	\end{eqnarray}
	\end{subequations}
with $\boldsymbol\mu=[\boldsymbol\mu_1,\ldots,\boldsymbol\mu_L]^T$ and $\boldsymbol\nu=[\boldsymbol\nu_1,\ldots,\boldsymbol\nu_L]^T$.
	
Clearly, \eqref{eq:MOTmulti} is akin to \eqref{eq:MOTform}. We now utilize the proximal gradient descent to solve it. Denote the set of $\bM$ satisfying the constraints \eqref{eq:MOTmulti2} by $\Pi(\boldsymbol\mu,\boldsymbol\nu)$ and $F(\bM) =\langle \bC(\bM), \bM\rangle$, then each iteration of the proximal gradient descent reads
	\begin{eqnarray}\nonumber
		\bM_{k+1} &=& \argmin_{\bM\in \Pi(\boldsymbol\mu,\boldsymbol\nu)} \langle \nabla F(\bM_k)-\frac{1}{\eta}\log \bM_k, \bM\rangle 
		\\&&- (\epsilon+\frac{1}{\eta}) \cH(\bM). \label{eq:Mkiter}
	\end{eqnarray}
\begin{prop}\label{prop:nablaF3}
The gradient of $F(\bM) = \langle \bC(\bM), \bM\rangle$ with $\bC$ in \eqref{eq:Cmulti} is
	\begin{equation}\label{eq:nablaFM}
		\nabla F(\bM) = \bC(\bM) + E(\bM)
	\end{equation}
where $E(\bM)(\ell,x_0,x_1,\ldots,x_T) = \sum_{i=0}^{T-1} E_i(\ell,x_i)$ with 
	\begin{eqnarray}\label{eq:Eil}
		&&E_{i} (\ell,y) = \sum_{m,x_i,x_{i+1}}\nabla W_{\ell m}(x_i-y)^T[x_{i+1}-x_i+
		\\&&\hspace{-0.4cm}\frac{1}{T} \sum_n\nabla W_{\ell n} * P_{-1,i}(\bM)-\frac{1}{T} b_\ell(x_i)]P_{-1,i,i+1}(\bM)(\ell,x_i,x_{i+1}).\nonumber
	\end{eqnarray}
\end{prop}
\begin{proof}
By definition, 
	\begin{eqnarray*}
		&&\hspace{-0.7cm}\langle \nabla F(\bM), \delta \bM\rangle \!=\! \langle \bC(\bM),\! \delta \bM\rangle\!+\!\!\! \sum_{i=0}^{T-1} T\langle  \frac{1}{T} (\sum_m\nabla W_{\ell m}\! *\! P_{-1,i}(\delta \bM)\!)^T
		\\&&(x_{i+1}-x_i+\frac{1}{T} \sum_n\nabla W_{\ell n} * P_{-1,i}(\bM))-\frac{1}{T} b_\ell(x_i),\bM\rangle. 
	\end{eqnarray*}
The second term on the right hand side equals
	\begin{eqnarray*}
		&&\hspace{-0.7cm}\sum_{i=0}^{T-1} \langle  (\sum_m\nabla W_{\ell m} * P_{-1,i}(\delta \bM))^T(x_{i+1}-x_i
		\\&&+\frac{1}{T} \sum_n\nabla W_{\ell n} * P_{-1,i}(\bM))-\frac{1}{T} b_\ell(x_i),P_{-1,i,i+1}(\bM)\rangle
		\\&=& \sum_{i=0}^{T-1}\langle  \delta \bM, E_{i} (\ell,x_i) \rangle
	\end{eqnarray*}
where $E_i$ is as in \eqref{eq:Eil}.
%	\begin{equation*}
%		E_{i} (\ell,y) = \sum_{m,x_i,x_{i+1}}\nabla W_{\ell m}(x_i-y)^T[x_{i+1}-x_i+\frac{1}{T} \sum_n\nabla W_{\ell n} * P_{-1,i}(\bM)-\frac{1}{T} b_\ell(x_i)]P_{-1,i,i+1}(\bM)(\ell,x_i,x_{i+1}).
%	\end{equation*}
Hence,
	\begin{equation*}
		\langle \nabla F(\bM), \delta \bM\rangle = \langle \bC(\bM), \delta \bM\rangle + \langle E(\bM), \delta \bM\rangle,
	\end{equation*}
with $E(\bM)(\ell,x_0,x_1,\ldots,x_T) = \sum_{i=0}^{T-1} E_i(\ell,x_i)$, which completes the proof.
\end{proof}

Plugging \eqref{eq:nablaFM} into \eqref{eq:Mkiter} yields
	\begin{eqnarray}\nonumber
		\bM_{k+1}\!\!\! &=&\!\!\! \argmin_{\bM\in \Pi(\boldsymbol\mu,\boldsymbol\nu)} \langle \bC(\bM_k)+E(\bM_k)-\frac{1}{\eta} \log \bM_k, \bM\rangle  
		\\&&\!\!\!-(\epsilon + \frac{1}{\eta})\cH(\bM).\label{eq:proximaliterM}
	\end{eqnarray}
Apparently, both $\bC(\bM)$ in \eqref{eq:Cmulti} and $E(\bM)$ in \eqref{eq:nablaFM} have a graphical structure associated with the graph shown in Figure \ref{fig:Multispecies}. Assume $\bM_k$ shares the same graphical structure, then the solution $\bM_{k+1}$ to \eqref{eq:proximaliterM} also has this graphical structure. Thus, with proper initialization, the graphical structure (Figure \ref{fig:Multispecies}) is preserved through the iteration \eqref{eq:proximaliterM}. 
Each iteration \eqref{eq:proximaliterM} is a graphical OT problem and can be solved using a (generalized) SBP algorithm. Thus, the Proximal Sinkhorn Belief Propagation algorithm (Algorithm \ref{alg:proximalsinkhorn}) is applicable to the density control problem with multiple species as long as the SBP subroutine is tailored for the graphical structure in Figure \ref{fig:Multispecies}.
	\begin{figure}[h]
	\centering
	\begin{tikzpicture}[scale=0.95, every node/.style={scale=0.95}]
	\footnotesize
	%\scriptsize
	 \tikzstyle{bluenode}=[draw, circle, minimum size=1.8em, inner sep=1pt, fill=blue!25];
	\tikzstyle{main}=[circle, minimum size = 9mm, thick, draw =black!80, node distance = 10mm]
	\tikzstyle{obs}=[circle, minimum size = 9mm, thick, draw =black!80, node distance = 10 mm and 6mm ]
	\tikzstyle{edge}=[>=stealth, thick, black]
	%\node[main,fill=black!10] (mu0) {$\mu_0$};
	\node[main,fill=blue!25] (mu0) {$\ell$};
	
	\node[] (phi1c) [below=of mu0] {};  
	\node[obs,fill=blue!1] (mu2) [left=of phi1c] {$x_1$};  
	\node[obs,fill=blue!25] (mu1) [left=of mu2] {$x_0$};  
	\node[obs,fill=blue!1] (muS1) [right=of phi1c] {$\!x_{\!T\!-\!1\!}\!$};
	\node[obs,fill=blue!25] (muS) [right=of muS1] {$x_{T}$};
	
	\draw[edge] (mu0) to (mu1);
	\draw[edge] (mu0) to (mu2);
	\draw[edge] (mu0) to (muS1);
	\draw[edge] (mu0) to (muS);
	\draw[edge] (mu1) to (mu2);
	\draw[edge] (muS1) to (muS);
%	\draw[->, -latex, thick] (mu0) -- node[left] {$P_{-1,0}(\bM)=R^{(-1,0)}$} (mu1);
%	\draw[->, -latex, thick] (mu0) -- node[left] {$C_1$} (mu2);
%	\draw[->, -latex, thick] (mu0) -- node[left] {\ $C_{T-1}$} (muS1);
%	\draw[->, -latex, thick] (mu0) -- node[right] {\ $P_{-1,T}=R^{(-1,T)}$} (muS);
%	\draw[->, -latex, thick] (mu1) -- node[below] {$C$} (mu2);
%	\draw[->, -latex, thick] (muS1) -- node[below] {$C$} (muS);
	\draw[loosely dotted, very thick] (mu2) -- (muS1); 
	\end{tikzpicture}

	\caption{Graph for graphical OT \eqref{eq:proximaliterM}}
	\label{fig:Multispecies}
	\end{figure}
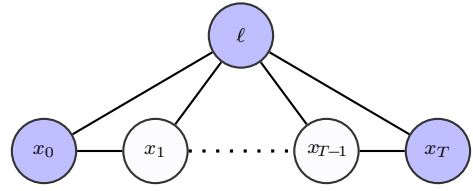
%\begin{algorithm*}[tb]
%   \caption{Proximal Sinkhorn Belief Propagation algorithm}
%   \label{alg:proximalsinkhornbelief}
%\begin{algorithmic}
%   \STATE Input: cost tensor $\bC$, regularization $\epsilon$, stepsize $\eta$, number of iterations $K$
%   \STATE Initialize $\bM_1$ to be a uniform probability vector
%%   \WHILE{not converged}
%   \FOR{$k = 1, 2, 3, 
%   \ldots,K$}
%        \STATE Compute $\bC(\bM_k)+E(\bM_k)-\frac{1}{\eta} \log \bM_k$
%        \STATE Solve \eqref{eq:proximaliterM} using the (generalized) Sinkhorn Belief Propagation algorithm to obtain $\bM_{k+1}$
%    \ENDFOR
%%    \ENDWHILE
%\end{algorithmic}
%\end{algorithm*}
\begin{remark}
In the multi-species setting, the SBP algorithm is applied to the graphical OT in Figure \ref{fig:Multispecies}. The computation complexity for each outer iteration becomes $\cO(D^2 LT)$.
\end{remark}
\begin{remark}
Even though in \eqref{eq:MSfluid} the object cost is decoupled into $L$ separate terms, each corresponds to one species, it is straightforward generalize the method to include cost such as $\int_0^1\int V(x) \rho_t dx dt$ that depends on the group behavior of all the individuals. 
\end{remark}
%The major difference between the above two formulation is that the optimization variable in the latter is a single $T+2$ dimensional tensor while the optimization variable consists of $L$ copies of $T+1$ dimensional tensors. 
%combine $\bM^\ell$ into a single tensor. discuss graphical structure of the multi-species problem

\section{Linear quadratic cases}\label{sec:LQ}
%not needed, maybe present the results for multiple species
A special case of particular interest is the linear quadratic density control problem where the dynamics of the individuals are linear and the costs are quadratic. That is, in the linear quadratic setting, $b_\ell, W_{m\ell}, V_\ell$ are of the form
	\begin{subequations}\label{eq:quad}
	\begin{equation}
		b_\ell(x) = A_\ell x,
	\end{equation}
	\begin{equation}
		W_{m\ell}(x) = \frac{1}{2} x^T\bar A_{m\ell} x,~\mbox{with}~\bar A_{m\ell} = \bar A_{m\ell}^T  = \bar A_{\ell m},
	\end{equation}
and 
	\begin{equation}
		V_\ell(x) = \frac{1}{2} x^T Q_\ell x.
	\end{equation}
	\end{subequations}
rendering linear dynamics for each individual
	\begin{eqnarray}\nonumber
		dX_{\ell,t}^i &=& -\frac{1}{N} \sum_{m} \sum_j \bar A_{\ell m}(X_{\ell,t}^i -X_{m,t}^j) dt + A_\ell X_{\ell,t}^idt 
		\\&&+ \sigma(u_{\ell,t}^idt+\sqrt{\epsilon} dB_{\ell,t}^i),\quad i = 1,\ldots, N,\label{eq:lineardyn}
	\end{eqnarray}
and quadratic cost in the mean field limit
	\begin{equation}
		\sum_{\ell=1}^L\int_0^1 \int_{\mR^d} [\frac{1}{2}\|\xi_{\ell,t}(x)\|^2 + \frac{1}{2}x^TQ_\ell x] \rho_{\ell,t}(x) dx dt.
	\end{equation}

When the feedback strategies $\xi_{\ell,t}$ are linear and the initial distributions are Gaussian, the distribution $\rho_{\ell,t}, \ell =1,\ldots, L$ of the populations remain Gaussian all the time. Thus, we assume the marginal distributions are Gaussian, denoted by
	\[
		\mu_\ell = \cN(m_\ell^0,\Sigma_\ell^0),\quad \nu_\ell = \cN(m_\ell^1,\Sigma_\ell^1),~\ell = 1,\ldots, L.
	\]
When there is no interaction between the individuals, the problem reduces to the covariance control problem \cite{CheGeoPav14a}.	The coupling of the agents introduces extra complexities. Recently, the linear quadratic density control problem for one species ($L=1$) has been addressed in \cite{CheGeoPav18a}.
 
We next present the solution when multiple species are involved. To this end, we parametrize the Gaussian distributions $\rho_\ell$ by 
	\begin{equation}\label{eq:LQrho}
		\rho_{\ell,t} = \cN(m_{\ell}(t), \Sigma_{\ell}(t)).
	\end{equation}
Just as standard linear quadratic optimal control, the Lagrange multipliers $\lambda_1,\lambda_2,\ldots,\lambda_L$ in \eqref{eq:optimality} are quadratic, denoted by
	\begin{equation}\label{eq:LQlambda}
		\lambda_\ell(t,x) = -\frac{1}{2} x^T \Pi_\ell(t) x + n_\ell(t)^T x + c_\ell(t).
	\end{equation}
Plugging \eqref{eq:quad}, \eqref{eq:LQrho} and \eqref{eq:LQlambda} into the optimality condition \eqref{eq:optimality} yields a coupled equation system (for all $\ell = 1,2,\ldots,L$)
	\begin{subequations}\label{eq:optimalLQ}
	\begin{eqnarray}\nonumber
		&&\hspace{-0.7cm} \dot \Pi_\ell \!-\!\Pi_\ell \sigma\sigma^T \Pi_\ell \!+\!Q_\ell \!+\!(A_\ell\!-\!\!\!\sum_m \bar A_{\ell m}\!)^T\Pi_\ell
		\\&&+ \Pi_\ell (A_\ell \!-\!\!\!\sum_m \bar A_{\ell m}\!) \!=\! 0 \label{eq:optimalLQ1}
		\\&&\hspace{-0.7cm} \dot\Sigma_\ell -(A_\ell \!-\!\!\sum_m \bar A_{\ell m} -\sigma\sigma^T\Pi_\ell) \Sigma_\ell \!-\!\!\Sigma_\ell(A_\ell -\!\!\sum_m \bar A_{\ell m} -\sigma\sigma^T\Pi_\ell)^T \nonumber
		\\&& -\epsilon \sigma\sigma^T = 0 \label{eq:optimalLQ2}
		\\&&\hspace{-0.7cm} \Sigma_\ell(0) = \Sigma_\ell^0,\quad \Sigma_\ell(1) = \Sigma_\ell^1,\label{eq:optimalLQ3}
		\\&&\hspace{-0.7cm} \dot n_\ell +(A_\ell-\sum_m \bar A_{\ell m} -\sigma\sigma^T\Pi_\ell )^T n_\ell+ \sum_m \bar A_{\ell m} n_m \nonumber
		\\&&- \sum_m(\Pi_\ell \bar A_{\ell m} +\bar A_{\ell m} \Pi_m) m_m= 0 \label{eq:optimalLQ4}
		\\&&\hspace{-0.7cm} \dot m_\ell \!-\! (A_\ell \!-\!\!\sum_m \bar A_{\ell m} \!-\!\sigma\sigma^T\Pi_\ell)m_\ell \!-\!\!\sum_m \bar A_{\ell m} m_m \!-\! \sigma\sigma^T n_\ell = 0 \label{eq:optimalLQ5}
		\\&&\hspace{-0.7cm} m_\ell(0)= m_\ell^0,\quad m_\ell(1) = m_\ell^1.\label{eq:optimalLQ6}
	\end{eqnarray}
	\end{subequations}
	
In the above, \eqref{eq:optimalLQ2} and \eqref{eq:optimalLQ5} are associated with the Fokker-Planck equation \eqref{eq:optimality2}. To see this, note that in the mean field limit each individual in the $\ell$-th species, under control policy $\sigma^T\nabla \lambda_\ell$, follows the dynamics
	\begin{eqnarray*}
		dX_t &=& (A_\ell -\sum_m \bar A_{\ell m} -\sigma\sigma^T\Pi_\ell) X_t dt +\sum_m \bar A_{\ell m}m_m dt 
		\\&&+\sigma\sigma^T n_\ell dt+\sqrt{\epsilon} \sigma dB_t.
	\end{eqnarray*}
For this linear dynamics, the Fokker-Planck equation \eqref{eq:optimality2} reduces to the Lyapunov equation \eqref{eq:optimalLQ2} for the covariance and a differential equation \eqref{eq:optimalLQ5} for the mean dynamics. The PDE \eqref{eq:optimality1} becomes \eqref{eq:optimalLQ1} and \eqref{eq:optimalLQ4}. In particular, \eqref{eq:optimalLQ1} is a Riccati equation.  

It turns out that \eqref{eq:optimalLQ} has a closed-form solution. First we observe that \eqref{eq:optimalLQ} is that $\Pi_\ell, \Sigma_\ell$ can be solved from \eqref{eq:optimalLQ1}-\eqref{eq:optimalLQ3} and are independent of the value of $n_\ell, m_\ell$. Moreover, the equations for $\Pi_\ell, \Sigma_\ell$ are independent to each other for different species $\ell$. Thus, each pair $\Pi_\ell, \Sigma_\ell$ can be computed separately. Note that the boundary conditions in \eqref{eq:optimalLQ3} for the coupled differential equations \eqref{eq:optimalLQ1}-\eqref{eq:optimalLQ3} are not conventional; the boundary values of $\Sigma_\ell$ are given on both end while no boundary value for $\Pi_\ell$ is provided. Nevertheless, closed-form solutions to \eqref{eq:optimalLQ1}-\eqref{eq:optimalLQ3} can be obtained. Let 
	\[
		\HH_\ell(t) =\epsilon\Sigma_\ell(t)^{-1}-\Pi_\ell(t),
	\]
then \eqref{eq:optimalLQ1}-\eqref{eq:optimalLQ3} become a coupled Riccati equation system
%	\begin{subequations}\label{eq:Riccati}
	\begin{eqnarray*}\label{eq:Riccati1}
		&&\hspace{-0.7cm} \dot \Pi_\ell \!-\!\Pi_\ell \sigma\sigma^T \Pi_\ell\! +\!Q_\ell\! +\!(A_\ell \!-\!\!\sum_m \bar A_{\ell m})^T\Pi_\ell \!+\! \Pi_\ell (A_\ell \!-\!\sum_m \bar A_{\ell m}) \!=\! 0
		\\&&\hspace{-0.7cm} \dot \HH_\ell \!+\!\HH_\ell \sigma\sigma^T \HH_\ell \!-\!Q_\ell \!+\!(A_\ell \!-\!\sum_m \bar A_{\ell m})^T\HH_\ell \!+\! \HH_\ell (A_\ell \!-\!\sum_m \bar A_{\ell m}) \!=\! 0 \label{eq:Riccati2}
		\\&&\hspace{-0.7cm} \Pi_\ell(0)+\HH_\ell(0) = \epsilon(\Sigma_\ell^0)^{-1},\quad \Pi_\ell(1)+\HH_\ell(1) = \epsilon(\Sigma_\ell^1)^{-1}.\label{eq:Riccati3}
	\end{eqnarray*}
%	\end{subequations}
This is exactly the characterization of the covariance control problem \cite{CheGeoPav14a} for the dynamics
	\begin{equation}\label{eq:covlinear}
		dX_t = (A_\ell -\sum_m \bar A_{\ell m})X_t dt + \sigma(u_t dt + \sqrt{\epsilon} dB_t).
	\end{equation}
Assume \eqref{eq:covlinear} is controllable, then the above Riccati equation system allows a unique solution in closed-form. We refer the reader to \cite{CheGeoPav14a} for the exact expression for the closed-form solution.
	
%\begin{prop}
%The solution to \eqref{eq:optimalLQ1}-\eqref{eq:optimalLQ3} are specified by the initial condition...
%\end{prop}

Once $\Pi_\ell, \Sigma_\ell$ are computed, we can plug them into \eqref{eq:optimalLQ4}-\eqref{eq:optimalLQ6} to solve for $n_\ell, m_\ell$. These are standard linear equations and can be solved efficiently. 
%From the above derivation, we can conclude that, in the linear quadratic setting, the control of the covariance and mean are decoupled and can be addressed separately. 
Once the solution to \eqref{eq:optimalLQ1}-\eqref{eq:optimalLQ6} is obtained, we can recover the optimal control as
	\begin{equation}
		\xi_{\ell,t}(x) = -\sigma^T\Pi_\ell(t)x+\sigma^Tn_\ell(t),
	\end{equation}
which is a linear state feedback.
%It turns out that in this linear quadratic setting the control of the mean and covariance are decoupled. In the above, \eqref{eq:optimalLQ1}-\eqref{eq:optimalLQ2} are for the control of covariance only. They can be solved in closed-form (display the solution).

\section{Numerical examples}\label{sec:eg}
%one example in LQ setting and one example with general cost

In this section we provide several numerical examples to illustrate the proposed framework on density control of interacting agent systems. In the first example, we demonstrate the solution to the density control problem in the linear quadratic setting for both single species and multiple species can steer the agents to target distributions. In the second example, we illustrate how a general interacting particle system evolve from an initial configuration to a target configuration. 

\subsection{Linear quadratic density control}
We first consider $N$ agents of the same species interacting with each other through the dynamics \eqref{eq:lineardyn} with 
	\[
		A_1 = \left[\begin{matrix}
		0 & 1\\ 0 & 0
		\end{matrix}\right]
	\]
	\[
		\bar A_{11} = \left[\begin{matrix}
		0 & 0\\ 0 & 0.5
		\end{matrix}\right],
	\]
and $\sigma = [0~ 1]^T$. Each agent alone has a trivial second order dynamics. The choice of $\bar A_{11}$ corresponds to an interacting potential that synchronizes the velocities of the agents. The noise intensity $\epsilon$ is set to be $1$ and we choose $Q_1=I$. Our goal is to find a global feedback policy so that the distribution of the agents is transformed from 
	\[
		\mu_1 = \cN ( \left[\begin{matrix}
		1 \\ 1
		\end{matrix}\right], \left[\begin{matrix}
		0.25 & 0\\ 0 & 0.25
		\end{matrix}\right])
	\]
to 
	\[
		\nu_1 = \cN ( \left[\begin{matrix}
		1.5 \\ 0.8
		\end{matrix}\right],
		\left[\begin{matrix}
		0.5 & 0\\ 0 & 0.1
		\end{matrix}\right]).
	\]
Figure \ref{fig:LQsingle} depicts the $3-\sigma$ confidence level of the Gaussian distribution of the agents. The states of the agents should be inside this envelope with probability $99.73\%$. We also plot several typical trajectories of the agents, which stay inside the envelop as expected. 
	\begin{figure}[h]
	\centering
	\includegraphics[width=0.43\textwidth]{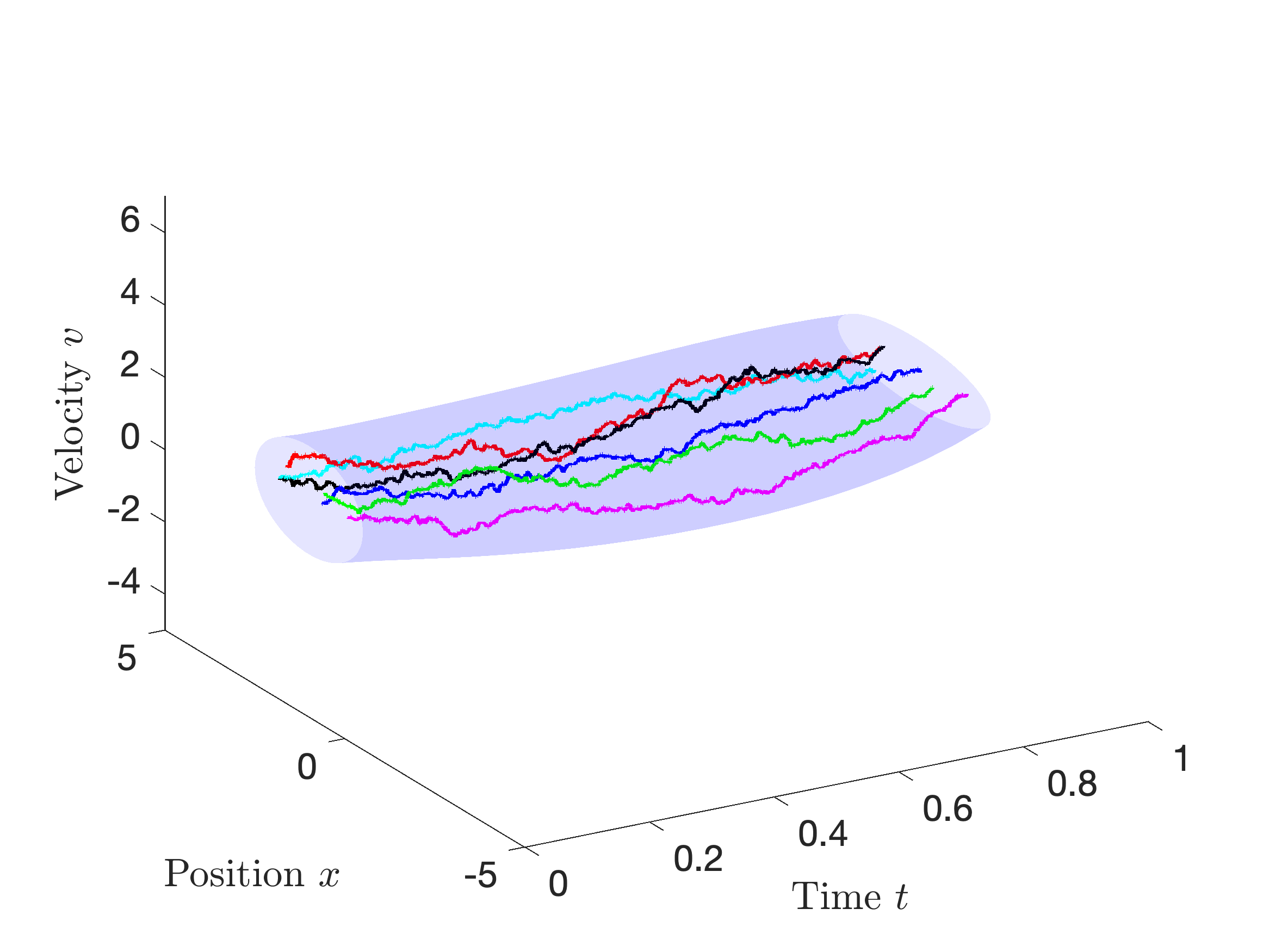}
	\caption{Covariance evolution of interacting agents under optimal control}
	\label{fig:LQsingle}
	\end{figure}

We next add another species to it. Let 
	\[
		A_2 = \left[\begin{matrix}
		0 & 1\\ 0 & 0
		\end{matrix}\right]
	\]
	\[
		\bar A_{22} = \left[\begin{matrix}
		0 & 0\\ 0 & 0.5
		\end{matrix}\right],
	\]
and the interaction matrix between the two species be 
	\[
		\bar A_{12} =\bar A_{21}= \left[\begin{matrix}
		-0.5 & 0\\ 0 & 0
		\end{matrix}\right].
	\]
The choice of $\bar A_{12}$ imposes a repulsive potential between the agents of the two species in position. 
Set $Q_2 = I$. The two marginal distributions of the second species is set to be 
	\[
		\mu_2 = \cN ( \left[\begin{matrix}
		-2 \\ -2
		\end{matrix}\right], \left[\begin{matrix}
		0.25 & 0\\ 0 & 0.25
		\end{matrix}\right])
	\]
and
	\[
		\nu_2 = \cN ( \left[\begin{matrix}
		-1 \\ -0.8
		\end{matrix}\right],
		\left[\begin{matrix}
		0.25 & 0\\ 0 & 0.1
		\end{matrix}\right]).
	\]
As before, we show the $3-\sigma$ confidence envelope of the distributions of both species in Figure \ref{fig:LQtwo}, together with some typical trajectories of the agents. It is clear from Figure \ref{fig:LQtwo} that the behavior of the first species is affected by the second one with the tendency to stay away from the second species. 
	\begin{figure}[h]
	\centering
	\includegraphics[width=0.43\textwidth]{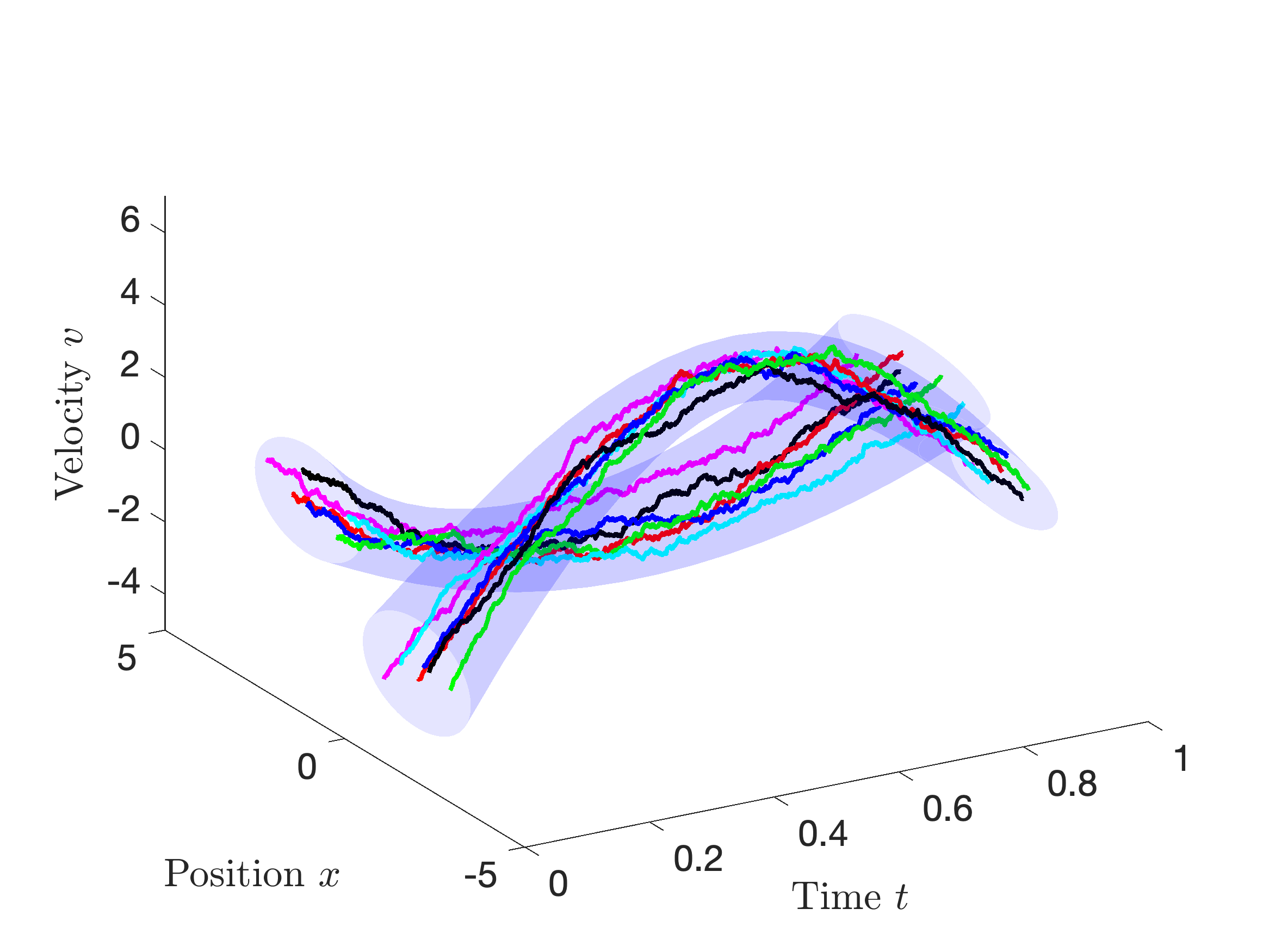}
	\caption{Covariance evolution of interacting agents from two species under optimal control}
	\label{fig:LQtwo}
	\end{figure}

\subsection{Density control for general dynamics}
In this example we consider $N$ agents in one dimensional space whose dynamics are described by \eqref{eq:particlegeneral} with
	\[
		W(x) = \frac{\beta}{|x|^\alpha},~b(x) = 0, ~\sigma = 1, ~\epsilon = 0.1.
	\]
That is, each agent alone is a first order integrator but they are affected by each other through the repulsive potential $W$.
Our goal is to steer the distribution of the agents from
	\[
		\mu = \cN(-0.4, 0.2)
	\]
to 
	\[
		\nu = \cN(0.4, 0.2)
	\]
in some optimal way. For simplicity, we set $V(\cdot) \equiv 0$. Thus the objective function is the total control effort.

The evolution of the agent distribution under optimal control policy is depicted in Figure \ref{fig:b0}-\ref{fig:a2b2} for different values of $\alpha$ and $\beta$. When there is no interaction ($\beta=0$) among the agents, the solution (Figure \ref{fig:b0}) corresponds to a standard Schr\"odinger bridge problem as expected. As we increase the repulsive potential (increase $\alpha$ or $\beta$), the agents spread out more quickly in between due to the repulsive force.
	\begin{figure}[h]
	\centering
	\includegraphics[width=0.43\textwidth]{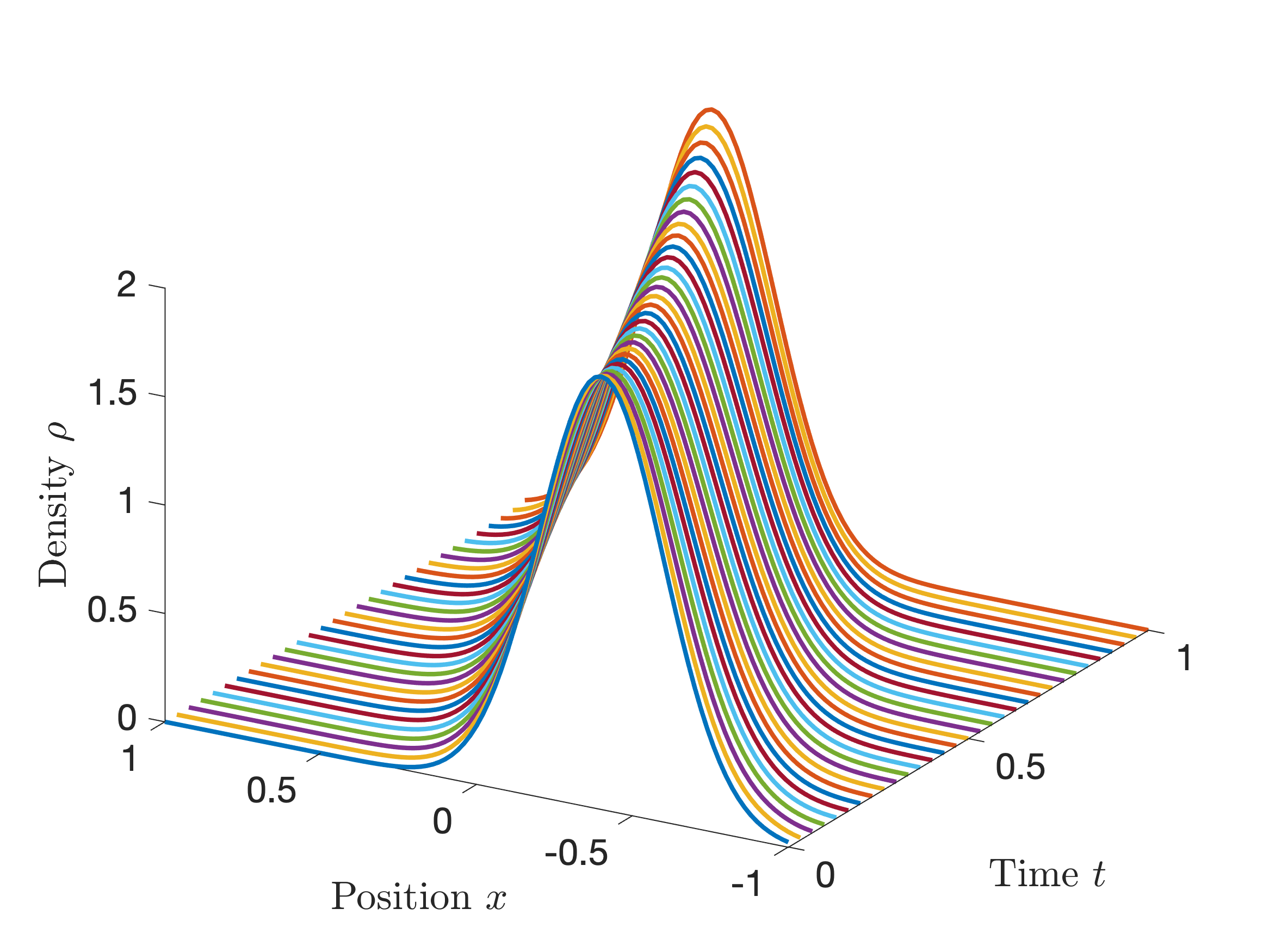}
	\caption{Density evolution under optimal control with $\beta = 0$}
	\label{fig:b0}
	\end{figure}
	\begin{figure}[h]
	\centering
	\includegraphics[width=0.43\textwidth]{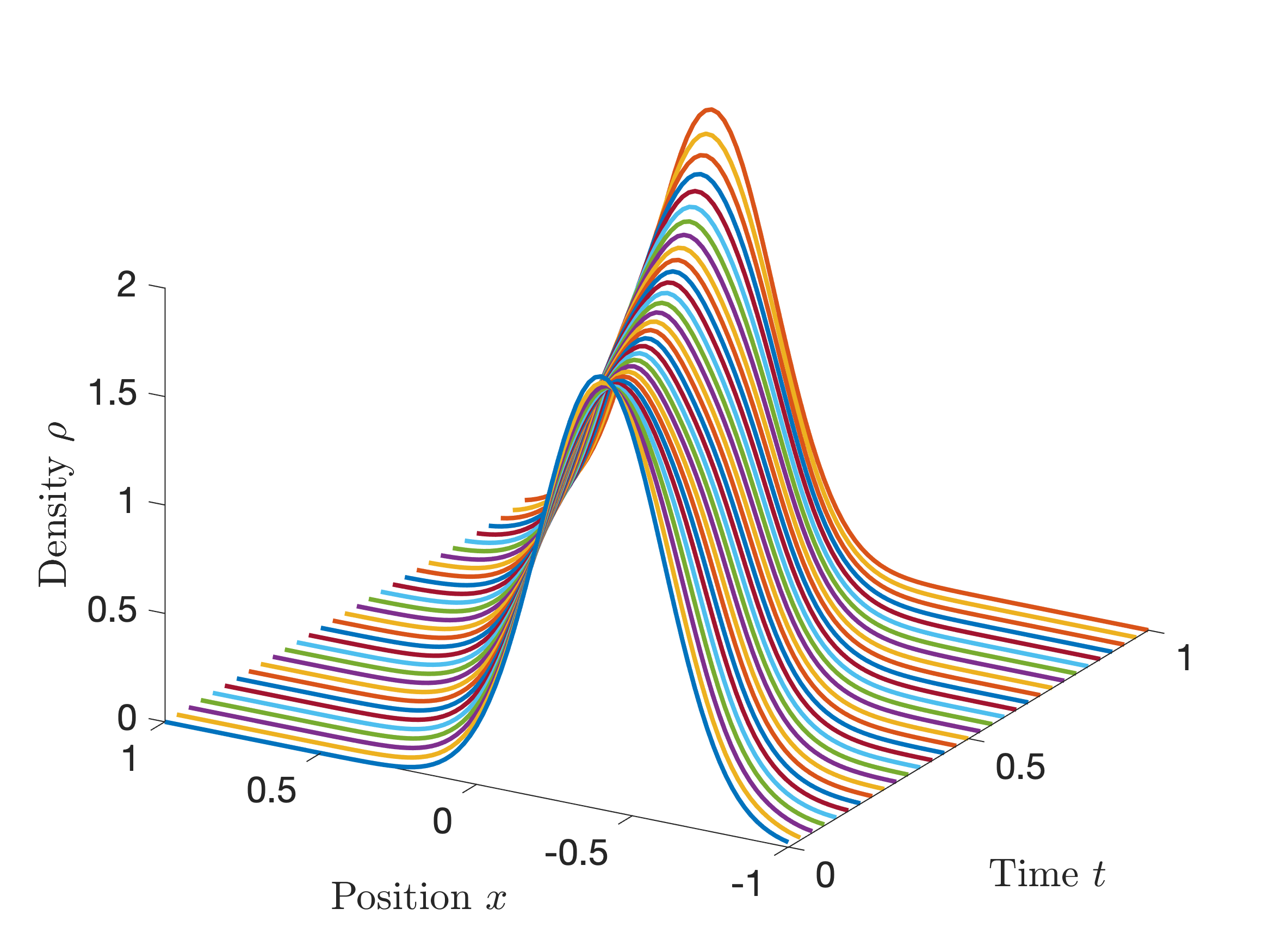}
	\caption{Density evolution under optimal control with $\alpha = 0.15, \beta = 1$}
	\label{fig:a15b1}
	\end{figure}
	\begin{figure}[h]
	\centering
	\includegraphics[width=0.43\textwidth]{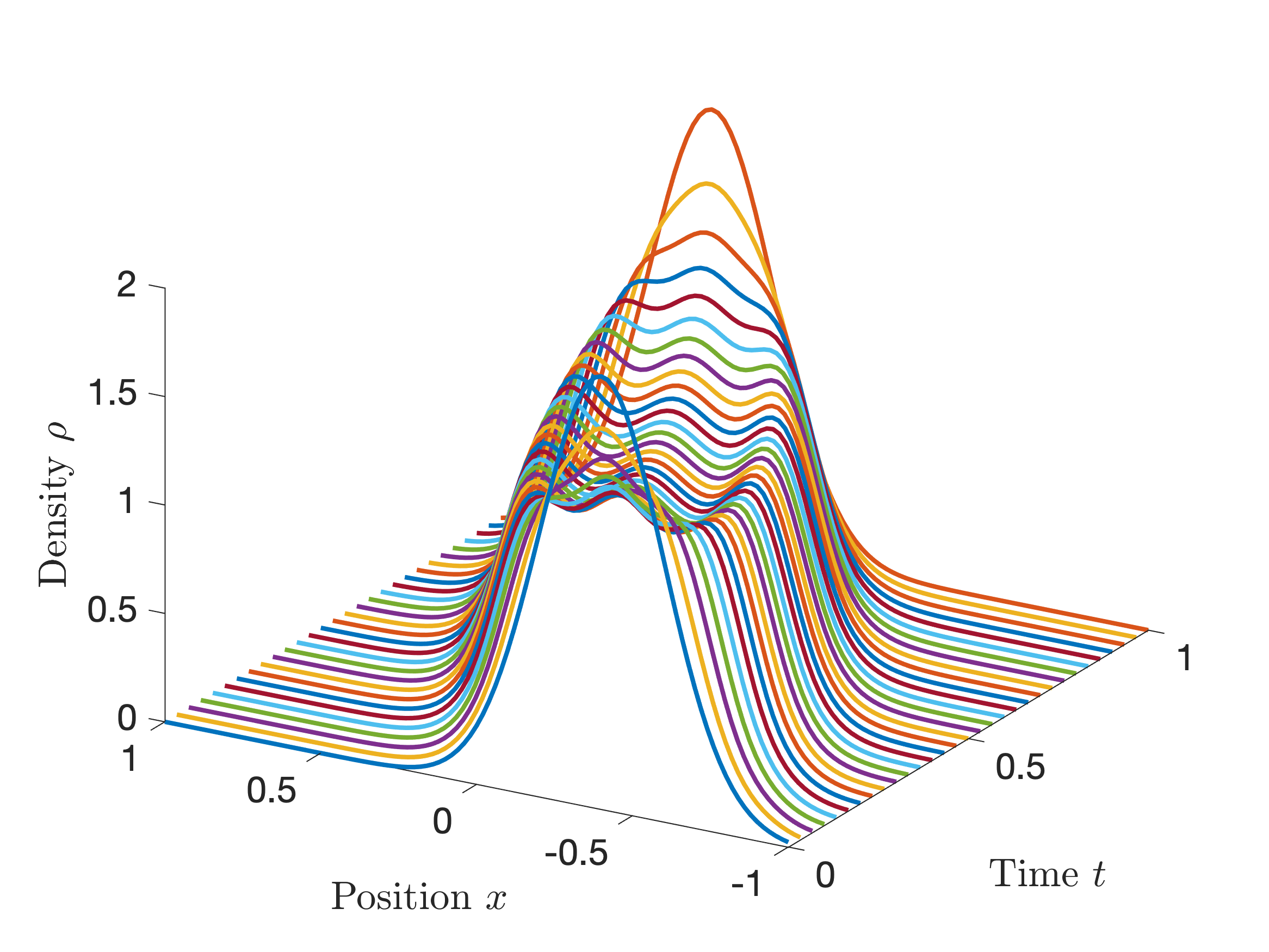}
	\caption{Density evolution under optimal control with $\alpha = 0.15, \beta = 2$}
	\label{fig:a15b2}
	\end{figure}
		\begin{figure}[h]
	\centering
	\includegraphics[width=0.43\textwidth]{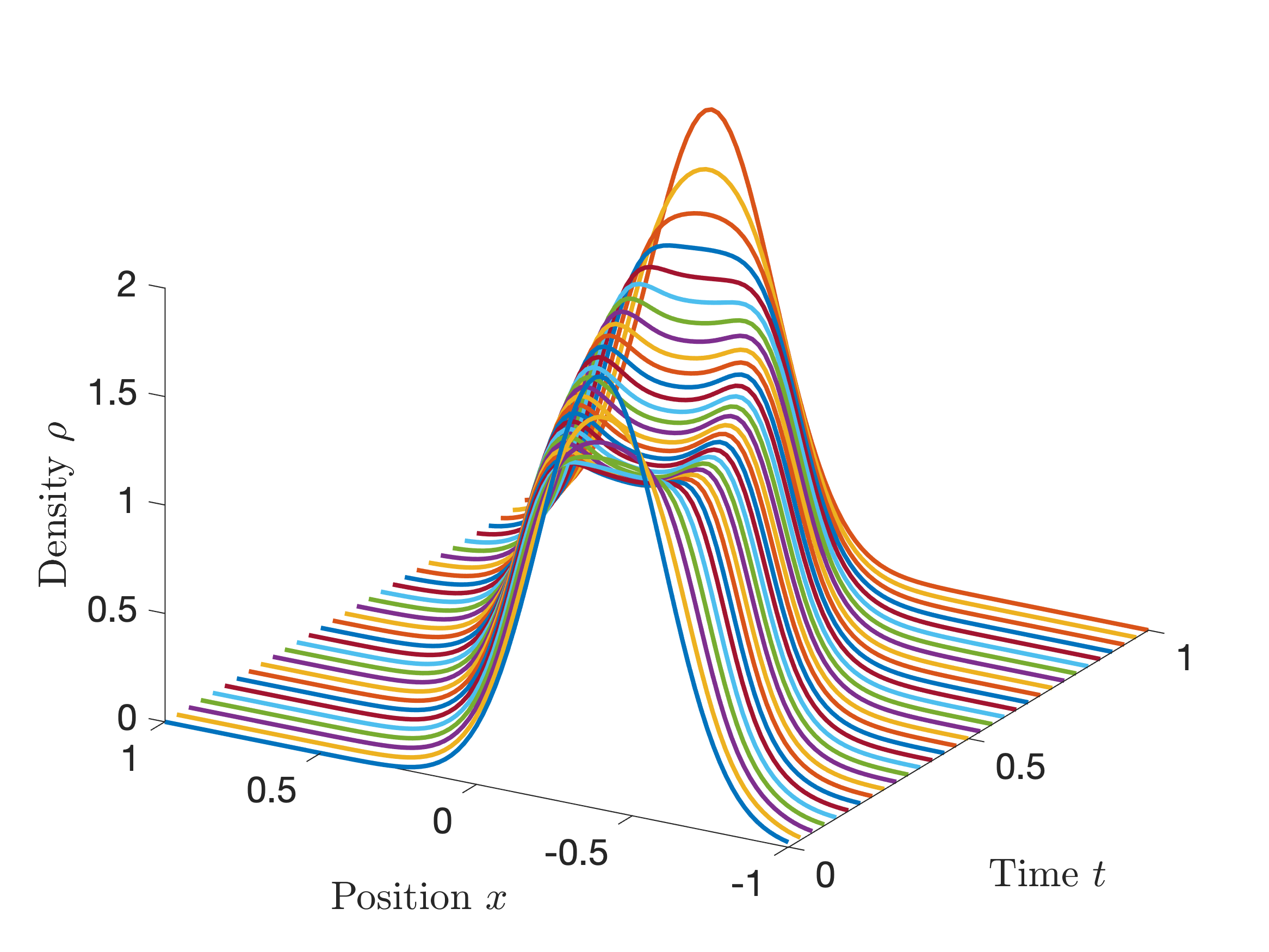}
	\caption{Density evolution under optimal control with $\alpha = 0.2, \beta = 1$}
	\label{fig:a2b1}
	\end{figure}
	\begin{figure}[h]
	\centering
	\includegraphics[width=0.43\textwidth]{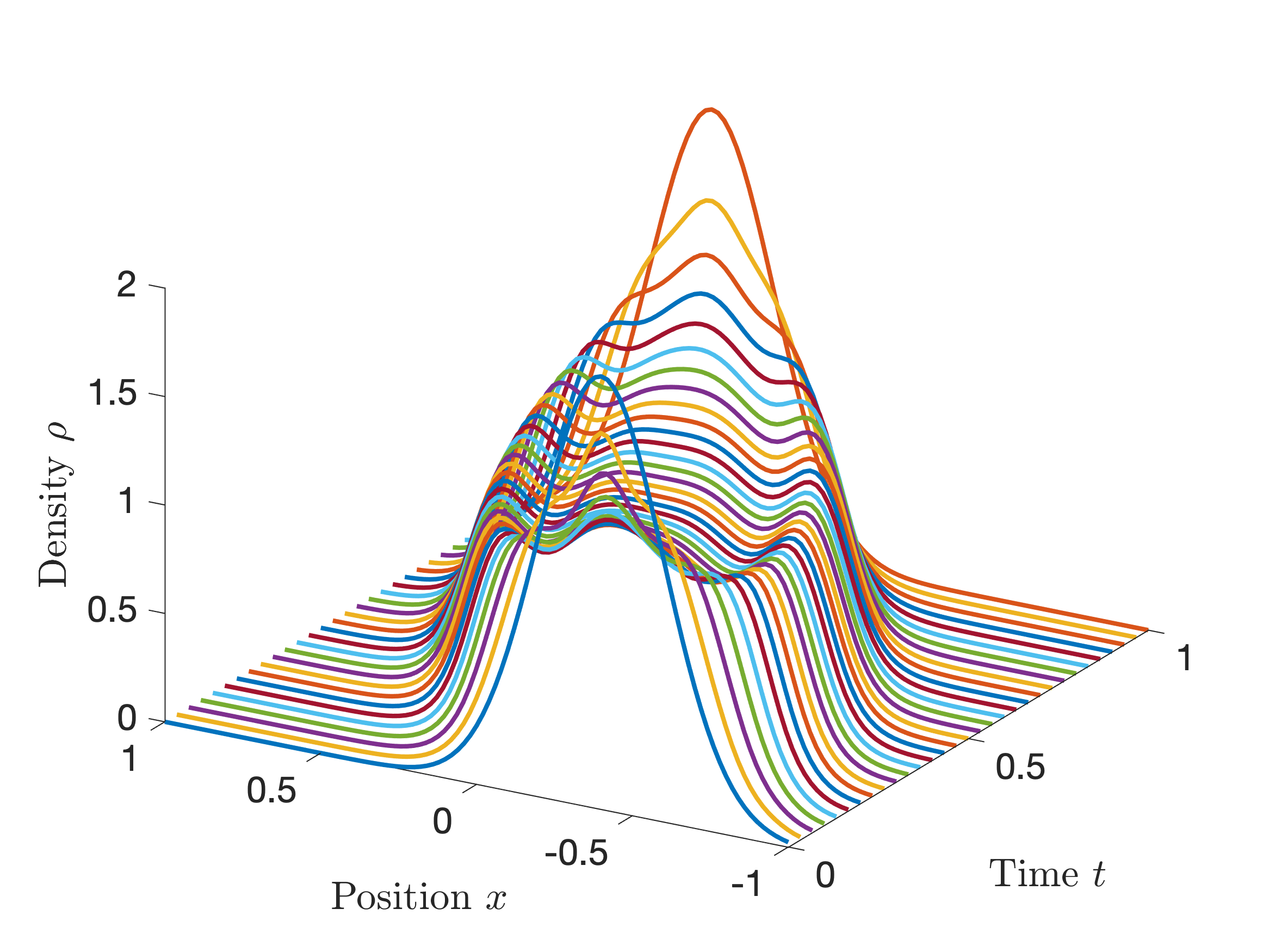}
	\caption{Density evolution under optimal control with $\alpha = 0.2, \beta = 2$}
	\label{fig:a2b2}
	\end{figure}

\section{Conclusion}\label{sec:conclusion}
%mention global nonlinear cost

We studied a swarm control problem for a large group of agents that are constantly interacting with each other. We considered the problem in the mean field and formulate it as a density control problem. We further reformulated it as a nonlinear multi-marginal optimal transport problem with proper discretization. Leveraging the proximal gradient framework, we were able to compute the solution via iteratively linearizing the problem and solving the linearized problems with the Sinkhorn belief propagation algorithm. We extended our framework to account for swarm control with multiple species. Finally, we provided a closed-form solution to the density control problem in the linear quadratic setting. In this work, we assume the total cost is the summation of the cost of each individual and is thus linear over the state distribution $\rho$. In the future we plan to extend our method to deal with cost that is a nonlinear function of $\rho$. 
	
{
\bibliographystyle{IEEEtran}
\bibliography{./refs}
}
%\begin{IEEEbiography}{Yongxin Chen}
%(S'13--M'17) received his BSc from Shanghai Jiao Tong University in 2011 and Ph.D. from University of Minnesota in 2016, both in Mechanical Engineering. He is currently an Assistant Professor in the School of Aerospace Engineering at Georgia Institute of Technology. He has served on the faculty at Iowa State University (2017-2018). He received the George S. Axelby Best Paper Award in 2017 for his joint work with Tryphon Georgiou and Michele Pavon. He received the NSF Faculty Early Career Development Program (CAREER) Award in 2020. His current research focuses on the intersection of control theory, machine learning, robotics and optimization.
%\end{IEEEbiography}
\end{document}